\documentclass[10pt]{amsart}
\usepackage{amssymb,amsfonts,amsmath,amsthm,amsopn,amstext,amscd,latexsym,xy,color,stmaryrd,epic,marginnote}
\usepackage{hyperref}
\usepackage{tikz,multicol,graphicx,longtable}
\usepackage{mathrsfs}
\usetikzlibrary{arrows,positioning,decorations.pathmorphing,
  decorations.markings
 }
 \usepackage{algorithm}
\usepackage{algorithmic}
  \usepackage{mathtools}
 \usepackage{nccmath}

 \newcommand*{\Scale}[2][4]{\scalebox{#1}{$#2$}}
\tikzset{inner sep=0pt, 
  root/.style={circle,draw,minimum size=7pt,thick}, 
  fatroot/.style={circle,draw,minimum size=10pt,thick}, 
  short root/.style={circle,fill,minimum size=7pt}, 
  doublearrow/.style={postaction={decorate}, 
  decoration={markings,mark=at position .7
  with {\arrow{angle 60}}},double distance=3pt,thick}
} 

\theoremstyle{plain}
\input xy
\xyoption{all}
\calclayout

\makeatletter 
 
\@addtoreset{algorithm}{section} 
\makeatother

\newtheorem{lemma}[algorithm]{Lemma}
\newtheorem{proposition}[algorithm]{Proposition}
\newtheorem{corollary}[algorithm]{Corollary}
\newtheorem{definition}[algorithm]{Definition}

\theoremstyle{remark}

\newtheorem{remark}[algorithm]{Remark}
\numberwithin{equation}{section}
\numberwithin{paragraph}{section}

\DeclareMathOperator{\Hom}{Hom}

\DeclareMathOperator{\Ad}{Ad}

\DeclareMathOperator{\Lie}{Lie}

\DeclareMathOperator{\Sp}{Sp}

\DeclareMathOperator{\diag}{diag}

\newcommand{\cF}{{\mathcal F}}

\newcommand{\frb}{{\mathfrak b}}

\newcommand{\frg}{{\mathfrak g}}

\newcommand{\frl}{{\mathfrak l}}

\newcommand{\frs}{{\mathfrak s}}

\newcommand{\frt}{{\mathfrak t}}

\newcommand{\GL}{\mathrm{GL}}

\newcommand{\PGL}{\mathrm{PGL}}
\newcommand{\SO}{\mathrm{SO}}
\newcommand{\SL}{\mathrm{SL}}

\newcommand{\al}{\alpha}
\newcommand{\be}{\beta}
\newcommand{\del}{\delta}

\newcommand{\varep}{\varepsilon}

\newcommand{\frakg}{\mathfrak g}
\newcommand{\frakt}{\mathfrak t}

\newcommand{\la}{\langle}
\newcommand{\ra}{\rangle}

\newcommand{\bR}{\mathbb{R}}
\newcommand{\bZ}{\mathbb{Z}}

\newcommand{\bG}{\mathbf{G}} 
\newcommand{\bT}{\mathbf{T}} 
\newcommand{\bN}{\mathbf{N}} 
\newcommand{\bB}{\mathbf{B}}
\newcommand{\bC}{\mathbb{C}}
\newcommand{\bA}{\mathbb{A}}
\newcommand{\bO}{\mathbb{O}}
\DeclareMathOperator{\Ht}{ht} 
\author{Beth Romano \and Jack A. Thorne}
\AtEndDocument{\bigskip{\footnotesize 
\textsc{Beth Romano}~~ \texttt{beth.romano@kcl.ac.uk}\\\textsc{Department of Mathematics, King's College London,
WC2R 2LS, UK}
\par
\textsc{Jack A. Thorne} ~~ \texttt{thorne@dpmms.cam.ac.uk}\\  \textsc{Department of Pure Mathematics and Mathematical Statistics, Wilberforce Road, Cambridge, CB3 0WB, UK}}}
\title{An LLL algorithm with symmetries}

\begin{document}

\maketitle

\begin{abstract}
We give a generalisation of the Lenstra--Lenstra--Lov\'asz (LLL) lattice-reduction algorithm that is valid for an arbitrary (split, semisimple) reductive group $\bG$. This can be regarded as `lattice reduction with symmetries'. We make this algorithm explicit for the classical groups $\bG = \Sp_{2g}$, $\SO_{2g}$, and for the exceptional group $\bG = G_2$. 
\end{abstract}


\section{Introduction}

Fix an integer $g \geq 1$. A lattice of rank $g$ is a free $\bZ$-module $\Lambda$ of rank $g$ together with an inner product $H$ on $\Lambda_\bR = \Lambda \otimes_\bZ \bR$. Lattice reduction refers to the problem of finding a $\bZ$-basis $b_1, \dots, b_g$ for $\Lambda$ that is optimal, in some sense. For example, a basis is said to be Minkowski reduced if for each $i = 1, \dots, n$, $(b_i, b_i)_H$ is minimal among vectors such that $b_1, \dots, b_{i-1}, b_i$ can be extended to a $\bZ$-basis of $\Lambda$. Finding a Minkowski-reduced basis is `hard': even finding $b_1$ (i.e.\ a shortest vector in $\Lambda$) is conjectured to be NP-hard \cite{Emd81, Mic02}.

Lattice reduction can be rephrased geometrically. Let us restrict to lattices $\Lambda$ that have covolume $\operatorname{covol}(\Lambda) = 1$. Giving a basis for a lattice $\Lambda$ is equivalent to giving an isomorphism $\Lambda \cong \bZ^g$, so we may as well fix $\Lambda = \bZ^g$ and consider only the set of inner products on $\bR^g$ such that $\bZ^g$ has covolume 1. The set of such inner products may be identified with the symmetric space $X_{\SL_g(\bR)} = \SO_g(\bR) \backslash \SL_g(\bR)$. There is a (right) action of the group $\SL_g(\bZ)$ on $X_{\SL_g(\bR)}$; under the given identification, acting by $\gamma \in \SL_g(\bZ)$ corresponds to making a change of basis (therefore replacing the Gram matrix $H$ by ${}^t \gamma H \gamma$).

The subset $\cF_g \subset X_{\SL_g(\bR)}$ corresponding to those inner products for which the standard basis of $\bZ^g$ is Minkowski reduced is a fundamental domain for the action of $\SL_g(\bZ)$, in the sense that the $\SL_g(\bZ)$-translates of $\cF_g$ cover $X_{\SL_g(\bR)}$, and no two distinct points in the interior of $\cF_g$ are $\SL_g(\bZ)$-conjugate. The problem of Minkowsi reduction, given a based lattice $(\Lambda, (b_i)_{i=1, \dots, g})$ corresponding to a point $x \in X_{\SL_g(\bR)}$, corresponds to the problem of finding an element $\gamma \in \SL_g(\bZ)$ such that $ x\gamma \in \cF_g$. 

We can give more efficient algorithms to find `optimal' bases if we relax the notion of optimality. The Lenstra--Lenstra--Lov\'asz (LLL) algorithm \cite{Len82} is a lattice reduction algorithm that, for any fixed $\delta \in (1/4, 1)$, returns a basis $b_1, \dots, b_g$ that is ($\delta$-)LLL reduced, in the following sense:
\begin{definition}\label{defn_standard_LLL}
Let $\del \in (1/4, 1)$, and let $H$ be an inner product on $\bR^g$. We say that $H$ is \emph{$\del$-reduced} (with respect to the standard basis $e_1, \dots, e_g$ of $\bZ^g$) if it satisfies the following conditions: first, let $e_1^\ast, \dots, e^\ast_g$ be the basis of $\bR^g$ obtained from applying Gram--Schmidt orthogonalisation to the basis $e_1, \dots, e_g$ with respect to the inner product $H$, therefore defined by the formulae
\[ e_j^\ast = e_j - \sum_{i=1}^{j-1} \mu_{i, j} e_i^\ast, \]
\[ \mu_{i ,j } = (e_j, e_i^\ast)_H / (e_i^\ast, e_i^\ast)_H \quad (1 \leq i < j \leq g). \]
Then:
\begin{enumerate} 
\item The basis $(e_i^\ast)_{i = 1}^g$ is size reduced: $| \mu_{i, j} | \leq 1/2$ for all $ 1\leq i < j \leq g$.
\item The basis $(e_i^\ast)_{i = 1}^g$ satisfies the Lov\'asz condition for each $i =1, \dots, g-1$: 
\begin{equation*}
(e_{i}^\ast, e_{i}^\ast)_H / (e_{i+1}^\ast, e_{i+1}^\ast)_H \leq (\delta - \mu_{i, i+1}^2)^{-1}.
\end{equation*}
\end{enumerate} 
\end{definition}
Although an LLL-reduced basis need not be Minkowski-reduced, it is close to optimal, for example in the sense that the lengths of the vectors in an LLL-reduced basis may be bounded in terms of the Minkowski successive minima of the lattice $(\bZ^g, ( \cdot, \cdot)_H)$ \cite[p. 518]{Len82}. It is also fast: for integer lattices (i.e. when $H$ has integer entries), it runs in polynomial time \cite[Proposition 1.26]{Len82}. As such, the LLL algorithm has had a very large number of applications in computational number theory (see e.g.\ \cite{Ngu10} for a detailed survey). 

The goal of this paper is to give a generalisation of the LLL lattice-reduction algorithm, where $\SL_g$ is replaced by an arbitrary (split, semisimple) reductive group $\bG$. This can be  regarded as `lattice reduction with symmetries'. We will define what it means for a point of $X_G = K \backslash \bG(\bR)$ (where $K$ is a maximal compact subgroup of $\bG(\bR)$) to be reduced, and give an algorithm that, given $x \in X_G$, produces $\gamma \in \bG(\bZ)$ such that $x \gamma$ is reduced. When $\bG = \SL_g$, our algorithm specialises to the one given in \cite{Len82}. 

Our initial motivation came from the paper \cite{Thorne}, where the second author defined a reduction covariant for representations $(\bG, V)$ arising from stably graded Lie algebras: more precisely, a $\bG(\bR)$-equivariant map 
\[ \mathcal{R} : V^s (\bR)\to X_\bG \]
from the (open, non-empty) set of stable vectors in $V(\bR)$ to the associated symmetric space. We can define an element of $V^s (\bR)$ to be reduced if its reduction covariant is reduced in the sense given in this paper, thus generalising the notion of a reduced genus-one normal curve in \cite{Cre10}.

In \cite{Thorne} we studied in more detail the representation of $\SO_{2g+1}$ associated to the stable $\bZ / 2 \bZ$-grading of a Lie algebra of type $A_{2g}$, which can be used to study 2-descent for the Jacobians of odd hyperelliptic curves of genus $g$, and gave an analogue of the LLL algorithm for the group $\SO_{2g+1}(\bZ)$. This was the impetus for the general algorithm given in this paper. We expect that the ideas developed in this paper will have applications (at least) to 2-descent for all of the families of curves associated to stably $\bZ / 2 \bZ$-graded Lie algebras in \cite{Tho13}. By way of illustration, we give an example related to 2-descent for even hyperelliptic curves of genus $g$ (and therefore the stable $\bZ / 2 \bZ$-grading of $\SL_{2g+2}$) in \S \ref{sec_case_SO_2g} below.

We will make our algorithm explicit in several cases. As already remarked, if $\bG = \SL_g$, then our algorithm (and definition of reducedness) is exactly the one given by Lenstra--Lenstra--Lov\'asz \cite{Len82}. The next case we consider is when $\bG = \Sp_{2g}$ is the symplectic group acting on $\bR^{2g}$. The associated symmetric space $X_{\Sp_{2g}}$ can then be identified with the Siegel upper half-space (in other words, the moduli space of based principally polarised abelian varieties of dimension $g$). This case is of interest as the reduction of period matrices of abelian varieties is useful for the evaluation of theta functions. In this case, our algorithm is essentially the one given in \cite{Dec04}, as a `cross product' of the LLL algorithm and the procedure given in \cite{Sie89}. (See also \cite{Gam06}.) We then go on to consider in more detail the cases of $\bG = \SO_{2g}$ and $\bG = G_2$, where our algorithm appears to be new. 

There are several directions in which one could hope to generalise our results further. We point out two natural questions. First, one could consider reduction with respect to the action of a not-necessarily-maximal subgroup $\Gamma \leq \bG(\bZ)$. It would be interesting to clarify the class of arithmetic groups $\Gamma$ to which our method applies (beginning with the group used in \cite{Thorne, Bha13}, which is \emph{not} the set of $\bZ$-points of a reductive group scheme).
Second, one could consider groups $\bG$ over $\bZ$ that are not necessarily split over $\bR$ (for example, special orthogonal groups that, over $\bR$, are not of signature $(g, g)$ or $(g, g+1)$). 

\subsection{Structure of the paper}

We begin with a split semisimple group $\bG$ over $\bZ$, and let $G = \bG(\bR)$. In \S \ref{s-set-up} we review some of the relevant theory of semisimple groups and define the relevant symmetric space $X_G$ that is the setting for our algorithm. In \S \ref{subsec_setup_in_case_SL_g}, we discuss the case when $\bG = \SL_g$ in terms of our more general set up. Motivated by this case, in \S \ref{s-reduced} we generalise to our general $\bG$ the definition of $\del$-reduced from the introduction.

In order to describe the general algorithm, we must understand the action of (lifts of) simple reflections on $X_G$: we describe their action explicitly in \S \ref{s-reflection}. Finally, in \S \ref{s-algorithm} we describe the general algorithm. 
In the subsequent sections, we make the algorithm explicit in the cases $\bG = \Sp_{2g}$ (\S \ref{sec_case_Sp_2g}), $\bG = \SO_{2g}$ (\S \ref{sec_case_SO_2g}), and $\bG = G_2$ (\S \ref{sec_case_G_2}). In the appendix, we describe the size-reduction step of the algorithm for a general group $\bG$.

\subsection{Notation}

We will generally use bold-face letters $\bG$ to denote group schemes over the integers, reserving Roman letters $G = \bG(\bR)$ for their groups of real points and gothic letters $\mathfrak{g} = \Lie \bG$ for their Lie algebras. Thus if $\bG = \SL_g$ then $G = \SL_g(\bR)$ and $\frg \leq M_n(\bZ)$ is the $\bZ$-Lie subalgebra of matrices of trace 0. We will consider $G$ as a real Lie group, endowed with its real topology, and write $G^\circ$ for the subgroup given by the connected component of the identity. We write $\bG_m$ for the multiplicative group and $\bG_a$ for the additive group. We will often refer to an inner product on $\bR^g$ by its Gram matrix $H$ with respect to the standard basis of $\bR^g$. We write the corresponding bilinear map on $\bR^g \times \bR^g$ as $(\cdot, \cdot)_H$.

\section{The general reduction algorithm}\label{s-general}

In this paper, we consider a semisimple Lie group $G$, its associated symmetric space $X_G$, and the action of an arithmetic group $\Gamma \leq G$ on $X_G$. We find the theory of reductive group schemes over $\bZ$ to be the most appropriate tool to do this. Having introduced a reductive group scheme $\bG$ over $\bZ$, we can take $\Gamma = \bG(\bZ)$, $G = \bG(\bR)$, and study the interplay between the action of elements of $\bG(\bZ)$ and the symmetric space $X_G$ using the usual structure theory. An excellent recent reference for this theory is the article \cite{Con14}.

Let us therefore take $\bG$ to be a split semisimple group over $\bZ$, with split maximal torus $\bT \leq \bG$. (For example, if $\bG = \SL_g$, then we may take $\bT$ to be the subgroup of diagonal matrices in $\bG$.) We will first recall some of the  structure theory of the pair $(\bG, \bT)$, and explain how the steps of the LLL lattice-reduction algorithm can be seen in these terms in the case $\bG = \SL_g$. We will then explain the algorithm for a general $\bG$.

\subsection{Set-up}\label{s-set-up}

We now review some of the structure theory of split semisimple groups. For futher details, see, e.g., \cite{Steinberg}.
Let $X^*(\bT) = \Hom(\bT, \bG_m)$ be the character group of $\bT$, and let $X_*(\bT) = \Hom(\bG_m, \bT)$ be its cocharacter group. The Lie algebra $\frg = \Lie \bG$ is a finite free $\bZ$-module that is semisimple as a representation of $\bT$, so we have a decomposition
\[ \frg = \frt \oplus \bigoplus_{\alpha \in X^\ast(\bT) - \{ 0 \}} \frg_\alpha, \]
where $\frg_\alpha$ denotes the eigenspace of the character $\alpha : \bT \to \bG_m$.
(This is the Cartan decomposition of the pair $(\bG, \bT)$.)
If an eigenspace $\frg_\alpha$ is non-zero, then it has rank 1. Characters for which $\frg_\alpha$ is non-zero are called roots, and we write $\Phi \subset X^\ast(\bT)$ for the set of roots.

Let $\al \in \Phi$, and let $X_\alpha \in \frg_\alpha$ be a basis vector (determined up to sign). Then there exists a unique pair $(Y_{\alpha}, \check{\alpha})$ consisting of a basis vector $Y_{\alpha} \in \frg_{-\alpha}$ and a cocharacter $\check{\alpha} \in X_\ast(\bT)$ such that $[X_\alpha, Y_{\alpha}] = d \check{\alpha}(1)$ and $\alpha \circ \check{\alpha}(t) = t^2$. (The element $\check{\alpha}$ is the coroot corresponding to $\alpha$.)  Further, there is a unique homomorphism $\phi_\alpha : \SL_2 \to \bG$ satisfying
\[ d \phi_\alpha\left( \begin{array}{cc} 0 & 1 \\ 0 & 0 \end{array}\right) = X_\alpha, \,\, d \phi_\alpha\left( \begin{array}{cc} 0 & 0 \\ 1 & 0 \end{array}\right) = Y_\alpha, \text{ and } \phi_\alpha\left( \begin{array}{cc} t & 0 \\ 0 & 1/t \end{array}\right) = \check{\alpha}(t).  \]
We write $u_\alpha : \bG_a \to \bG$ for the homomorphism defined by 
\[ u_\alpha(u) = \phi_\alpha \left( \begin{array}{cc} 1 & u \\ 0 & 1 \end{array}\right). \]
We refer to the image of $\phi_\al$ as a root $\SL_2$ and the image of $u_\al$ as a root group. Note that $\bT$ acts on each root group by conjugation, with $xu_\al(u)x^{-1} = u_\al(\al(x)u)$.

We define $s_\alpha = \phi_\alpha \begin{psmallmatrix} 0 & 1 \\ -1 & 0 \end{psmallmatrix} \in \bG(\bZ)$. This element normalises $\bT$, and its action on $\bT$ is given by the formula $s_\alpha x s_\al^{-1} =x / \check{\alpha}(\alpha(x))$.
The induced action on $X^\ast(\bT)$ is a reflection preserving $\Phi$. We also write $s_\al$ for the induced maps on $X^\ast(\bT)$ and $\Phi$. We remark that the element $s_\alpha \in \bG(\bZ)$ depends on the the choice of root vector $X_\alpha$, itself determined up to sign; the other choice would replace $s_\alpha$ by $\check{\alpha}(-1) s_\alpha$.

Let $G = \bG(\bR)$, $T = \bT(\bR)$, and let $A = T^\circ$ denote the connected component of the identity in $T$. We next describe the Iwasawa decomposition of $G$.
For more details see, e.g., \cite[Chapter IX]{Helgason} or \cite[Chapter VI]{Knapp}.
We first define a maximal unipotent subgroup $N$ of $G$. To do so, we choose a root basis $\Delta \subset \Phi$, which determines a system of positive roots $\Phi^+ \subset \Phi$. 
(Thus every element of $\Phi^+$ can be expressed as a $\bZ$-linear combination of elements of $\Delta$ with non-negative coefficients.) Then there is a unique (smooth, connected) subgroup $\bN \leq \bG$ whose Lie algebra equals $\bigoplus_{\alpha \in \Phi^+} \frg_\alpha$, and we take $N = \bN(\bR)$.

We next define a maximal compact subgroup $K$ of $G$.
Giving a maximal compact subgroup of $G$ is equivalent to giving a Cartan involution $\theta$ of $\bG_\bR$, i.e.\ an involution $\theta : \bG_\bR \to \bG_\bR$ such that the subgroup $\{ g \in \bG(\bC) \mid \theta(g) = \overline{g} \}$ is a maximal compact subgroup of $\bG(\bC)$. The associated maximal compact subgroup of $G = \bG(\bR)$ is then the fixed-point subgroup $G^\theta$. 
\begin{proposition}\label{prop_special_Cartan_involution} 
There is a unique involution $\theta_0 : \bG \to \bG$ with the following properties:
\begin{enumerate}
\item $\theta_{0, \bR} : \bG_\bR \to \bG_\bR$ is a Cartan involution.
\item $\theta_{0}$ normalises $\bT$.
\item For all $\alpha \in \Phi$, $d\theta_0(\frg_\alpha) = \frg_{-\alpha}$.  
\end{enumerate}
\end{proposition}
\begin{proof}
Fix a choice of basis vector $X_\alpha \in \frg_\alpha$ for each $\alpha \in \Delta$. According to \cite[Theorem 6.1.17]{Con14}, there is a unique involution $\theta_0 : \bG \to \bG$ that acts as inversion on $\bT$ and satisfies $d\theta_0(X_\alpha) = -Y_\alpha$ for each $\alpha \in \Delta$. In particular, $\theta_0$ normalises $\bT$. To check that $\theta_{0, \bR}$ is a Cartan involution, it is equivalent to check that the symmetric bilinear form $(X, Y)_0 = -\kappa(X, d\theta_0(Y))$ on $\frg_\bR$ is positive definite (where $\kappa$ is the Killing form). This is true: in fact, the decomposition 
\begin{equation}\label{eqn_real_Cartan_decomposition} \frg_\bR = \frt_\bR \oplus \bigoplus_{\alpha \in \Phi} \frg_{\alpha, \bR} 
\end{equation}
is orthogonal for $(\cdot, \cdot)_0$, and the restriction of $d\theta_0$ to each summand is positive definite. 

To show uniqueness, suppose that $\theta_1 : \bG \to \bG$ is another involution with the given properties. Then $\theta_0 \theta_1$ is an automorphism of $\bG$ that normalises $\bT$ and fixes each element of $\Phi$, so must have the form $\operatorname{Ad}(t)$ for some $t \in \bT^{ad}(\bZ)$, where $\bG^{ad}$ is the adjoint group (i.e. the quotient of $\bG$ by the centre $Z_\bG$) and $\bT^{ad} \leq \bG^{ad}$ is the image of $\bT$ in $\bG^{ad}$. Therefore $\theta_0 = \Ad(t) \circ \theta_1$. Moreover, the symmetric bilinear form $(X, Y)_1 = (\Ad(t^{-1}) X, Y)_0$ must again be positive definite. Since $\Ad(t)$ preserves the summands of the decomposition (\ref{eqn_real_Cartan_decomposition}), acting on each summand by multiplication by $\pm1$, $(\cdot, \cdot)_1$ can be positive definite only if $\Ad(t)$ acts by multiplication by $+1$ on each summand, or in other words if $\Ad(t) = 1$ and $\theta_0 = \theta_1$.
\end{proof}

Let $K = G^{\theta_0}$. The Iwasawa decomposition of $G$ (see \cite[Ch. VI, \S 4]{Knapp}) is then the statement that the multiplication map
\[ K \times A \times N \to G, (k, a, n) \mapsto kan, \]
is a diffeomorphism.

We define $X_G = K \backslash G$ to be the corresponding symmetric space for $G$. It is a (right) homogenous space for $G$ that is diffeomorphic to Euclidean space. Indeed, the Iwasawa decomposition allows us to identify $X_G$ with the product $A \times N$ (although we can't easily see the action of $G$ this way).

\subsection{Set-up for the case $\bG = \SL_g$}\label{subsec_setup_in_case_SL_g}

We now make all of the above explicit in the case $\bG = \SL_g$, with $\bT \leq \bG$ the subgroup of diagonal matrices. Then $\frs\frl_g = \Lie (\SL_g)$ is the Lie algebra of trace-zero $g \times g$ matrices. The roots are the homomorphisms 
\[ \alpha_{i, j} : \bT \to \bG_m, \diag(t_1, \dots, t_g) \mapsto t_i / t_j \]
for $1 \leq i,  j \leq g$, $i \neq j$. A  root basis is $\Delta = \{ \alpha_{1, 2}, \alpha_{2, 3}, \dots, \alpha_{g-1, g} \}$. The corresponding system of positive roots is $\Phi^+ = \{ \alpha_{i, j} \mid i < j \}$. With this choice, $N \leq \SL_g(\bR)$ is simply the group of upper-triangular matrices with $1$'s on the diagonal. The group $A$ is the group of diagonal matrices in $\SL_g(\bR)$ with positive entries. We choose the root vector $X_{\alpha_{i, j}}$ to be the elementary matrix $E_{i, j}$ (with $1$ in the $(i, j)$-entry and $0$'s elsewhere). Then the element $s_{\alpha_{i, j}} \in \SL_g(\bZ)$ corresponding to a root $\alpha_{i, j}$ is the matrix that (in terms of the standard basis $e_1, \dots, e_g$ of $\bR^g$) sends $e_i$ to $-e_j$, $e_j$ to $e_i$, and fixes the remaining basis vectors. 

The involution $\theta_0$ is defined by $\theta_0(g) = {}^t g^{-1}$. Then $K = \SO_g(\bR) = \{ x \in \SL_g(\bR) \mid {}^t x x = 1 \}$. We can identify $X_{\SL_g} = \SO_g(\bR) \backslash \SL_g(\bR)$ with the set of inner products $H$ on $\bR^g$ such that the lattice $\bZ^g$ has covolume 1. Indeed, there is a transitive right action of $\SL_g(\bR)$ on the set of such inner products (given by $H \cdot x = {}^t x H x$), and the stabiliser of the standard inner product $H_0$ is exactly the group $K$. 

In particular, the Iwasawa decomposition $X_{\SL_g} \cong A \times N$ implies that for any inner product $H$, there is a unique $a \in A$, $n  \in N$ such that $H = {}^t(an) (an)$. As is well-known, computing $a$ and $n$ corresponds to carrying out Gram--Schmidt orthogonalisation. More precisely, let $e_1^\ast, \dots, e_g^\ast \in \bR^g$ be the elements obtained after carrying out the Gram--Schmidt process with respect to $H$ on the standard basis $e_1, \dots, e_g$; then there are formulae $e_1 = e_1^\ast$ and 
\[e_j = e_j^\ast + \sum_{i=1}^{j-1} \mu_{i, j} e_i^\ast \]
($2 \leq j \leq g$) with $\mu_{i, j} = ( e_j, e_i^\ast)_H / (e_i^\ast, e_i^\ast)_H$.
We take $\mu_{i, i} = 1$ and $\mu_{i, j} = 0$ if $i > j$. We then have (cf. \cite[Ch. VI, \S 4, Example]{Knapp}):
\begin{lemma}\label{lem_Gram_Schmidt_for_SL_g}
Fix an inner product $H$. With notation as above, let $n = n_H = ( \mu_{i, j} )_{1 \leq i, j \leq g}$ and $a = a_H = \diag( \| e_i^\ast \|_H)_{1 \leq i \leq g}$. Then $a \in A$, $n \in N$, and $H = {}^t(an) an$. 
\end{lemma}
It follows that Definition \ref{defn_standard_LLL} (which expresses what it means for $H$ to be $\delta$-reduced) may be expressed in terms of the matrices $a$, $n$. This leads to the following lemma.
\begin{lemma}\label{lem_equivalence_with_LLL}
Let $\delta \in (1/4, 1)$. The inner product $H = {}^t(an) an$ is $\delta$-reduced in the sense of Definition \ref{defn_standard_LLL} if and only if the following conditions are satisfied:
\begin{enumerate}
\item $|n_{i, j}| \leq 1/2$ for all $1 \leq i < j \leq g$.
\item $\alpha_{i, i+1}(a)^2 \leq (\delta - n_{i, i+1}^2)^{-1}$ for each $i = 1, \dots, g-1$. 
\end{enumerate} 
\end{lemma}
In the algorithm described in \cite{Len82}, an inner product $H$ is moved to a reduced one by repeatedly applying the following two kinds of operations:
\begin{itemize}
\item Size-reduction: replacing $H$ by ${}^t (1 + m E_{i, j}) H (1 + m E_{i, j})$ for some $m \in \bZ$, $1 \leq i < j \leq g$.
\item Swapping basis vectors: replacing $H$ by ${}^t s_i H s_i$ for some $i = 1, \dots, g-1$, where $s_i = s_{\alpha_{i, i+1}}$.
\end{itemize}
These two operations correspond, respectively, to acting on $X_{\SL_g}$ on the right by elements of the form $\gamma = u_{\alpha_{i, j}}(m) \in \bN(\bZ)$, and by elements of the form $s_i$. It is easy to express the effect of the first operation in terms of the co-ordinates $(a, n)$: $(a, n)$ is replaced by $(a, n \gamma)$. The essential content of this step is that the set $\omega \leq N$ of matrices whose above-diagonal entries have absolute value at most $1/2$ is a fundamental set for the right action of $\bN(\bZ)$ on $N$. It is trickier to express the effect of swapping basis vectors in these co-ordinates. In \cite{Len82}, this is computed by looking at the effect on the Gram--Schmidt process for the new basis; we will see below that it suffices instead to do a computation in the root $\SL_2$ corresponding to the chosen simple root. 

\subsection{What it means to be reduced}\label{s-reduced}

We now return to the setting of \S \ref{s-set-up}, so $\bG$ is a split semisimple group with split maximal torus $\bT$. We fix a set $\Delta$ of simple roots as in \S \ref{s-set-up} and let $\bN$ be the corresponding unipotent subgroup.
We are nearly ready to say what it means for a point $H \in X_G$ to be reduced. This will depend on a choice of parameter $\delta \in (1/4, 1)$ and also a fundamental set $\omega \subset N$ for the right action of $\bN(\bZ)$. 

We first explain the properties that $\omega$ must have. By a fundamental set for the action of $\bN(\bZ)$, we mean a relatively compact subset $\omega \subset N$ such that $\omega \bN(\bZ) = N$. If we fix for each $\alpha \in \Delta$ a basis $X_\alpha$ for $\frg_\alpha$, then there is a unique homomorphism $p : N \to \bR^\Delta$ sending $u_\alpha(t)$ to the tuple $(t \delta_{\alpha \beta})_{\beta \in \Delta}$ for each $\alpha \in \Delta$. (Here $\del_{\al\be}$ is the Kronecker delta.)
Indeed, this follows from the fact that the group $N_+ := \la u_\al(t) \mid \al \in \Phi^+ - \Delta, t \in \bR\ra$ is normal in $N$, and every element of the quotient  $N/N_+$ can be written uniquely in the form $\prod_{\al \in \Delta} u_\al(t_\al)$ for some $t_\al \in \bR$ (cf. \cite[Corollaries 2 and 4 to Lemma 18]{Steinberg}). Note that $p$ does not depend on the ordering of the simple roots, since for simple roots $\al_i, \al_j$, the commutator $[u_{\al_i}(t_i), u_{\al_j}(t_j)] \in N_+$ (cf. \cite[Lemma 15]{Steinberg}). We write $p_\alpha$ for the composite of $p$ with projection to the $\alpha$-factor (so $p_\alpha(\prod_{\al \in \Delta} u_\al(t_\al)) = t_\alpha$).
For example, if $\bG = \SL_g$ and $X_{\alpha_{i, i+1}} = E_{i, i+1}$ for each $i = 1, \dots, g-1$, then $p$ is the homomorphism that takes a unipotent upper-triangular matrix to its $(g -1)$-tuple of super-diagonal entries. 

We say that a fundamental set $\omega$ is \textit{allowable} if $p(\omega) \subset [-1/2, 1/2]^\Delta$. This condition does not depend on the choice of basis vectors $X_{\alpha}$ (which are determined up to sign). Allowable fundamental sets exist, and can be constructed using the structure theory of the group $\bG$ (see the remarks after Algorithm \ref{alg_LLL_for_G} and the appendix). In particular examples, there is often a co-ordinate system leading to a natural choice (such as the one described for $\SL_g$ in \S \ref{subsec_setup_in_case_SL_g} above).
\begin{definition}\label{defn_general_definition_of_reducedness}
Let $\delta \in (1/4, 1)$, and let $\omega \subset N$ be an allowable fundamental set. We say that a point $x = K a n \in X_G$ is \emph{$(\delta, \omega)$-reduced} if it satisfies the following conditions:
\begin{enumerate} \item $n$ is size-reduced: $n \in \omega$.
\item $a$ satisfies the $\bG$-Lov\'asz condition: for each simple root $\alpha \in \Delta$, $\alpha(a)^2 \leq (\delta - p_\alpha(n)^2)^{-1}$. 
\end{enumerate}
\end{definition}
We give some initial evidence that this is a good notion. First, if $\bG = \SL_g$ (and $\omega$ is the standard choice, described in \S \ref{subsec_setup_in_case_SL_g}) then it specialises to the condition given in \cite{Len82}, by Lemma \ref{lem_equivalence_with_LLL}. Second, for a general group, we have the following weak statement:
\begin{proposition}
Let $\delta, \omega$ be as in Definition \ref{defn_general_definition_of_reducedness}. Then there is an integer $C = C(\delta, \omega) > 0$ such that each $G(\bZ)$-orbit in $X_G$ contains at most $C$ reduced elements. 
\end{proposition}
\begin{proof}
We deduce this from the reduction theory developed by Borel and Harish-Chandra as described in \cite{Borel}. Recall that a Siegel set is a subset of $G$ of the form $\mathfrak{S}_{c, \Omega} = K A_c \Omega$, where:
\begin{itemize} 
\item $c > 0$ and $A_c = \{ t \in A \mid \forall \alpha \in \Delta, \alpha(t) \leq c \}$.
\item $\Omega \subset N$ is a relatively compact subset. 
\end{itemize} 
It is known that for any Siegel set $\mathfrak{S} = \mathfrak{S}_{c, \Omega}$, the set $S_{c, \Omega} = \{ \gamma \in \bG(\bZ) \mid \mathfrak{S} \cap \mathfrak{S} \gamma \neq \emptyset \}$ is finite. Moreover, one can choose $c, \Omega$ so that $G = \mathfrak{S}_{c, \omega} \bG(\bZ)$. 
Here we observe that if $H \in X_G$ is $(\del, \omega)$-reduced, then it lies in the Siegel set $\mathfrak{S}_{c, \omega}$, where $c = (\delta - 1/4)^{-1/2}$. We see that the proposition holds with $C = \# S_{c, \omega}$.
\end{proof}

We remark that, as follows from the properties used in the proof, the Siegel sets $\mathfrak{S}_{c, \Omega}$ (and their projections to $X_G$) are a class of subsets that are `close to fundamental domains for the action of $\bG(\bZ)$'. 

One can ask for more refined statements. For example, it is shown in \cite{Len82} that if $H$ is an inner product on $\bR^g$ corresponding to a $(3/4)$-reduced point of $X_{\SL_g}$, then the standard basis vectors $e_i$ satisfy
\[ 2^{1-i} \lambda_{i, H} \leq (e_i, e_i)_H \leq 2^{g-i} \lambda_{i, H} \]
for each $i = 1, \dots, g$, where $\lambda_{i, H}$ is the $i^\text{th}$ Minkowksi successive minimum of the lattice $(\bZ^g, (\cdot, \cdot)_H)$. In particular, the length of $e_1$ is bounded by a constant (i.e. depending only on $g$) multiple of the length of a shortest vector. In the cases we consider where $\bG \leq \SL_g$ is a classical group (therefore with a defining $g$-dimensional representation), points of $X_G$ can be interpreted as inner products on $\bR^g$ satisfying additional symmetries (see Propositions \ref{prop_interpretation_of_symmetric_space} and \ref{prop_interpretation_of_symmetric_space_case_SO_2g} below), and one can prove analogous statements. 

\subsection{The action of a simple reflection}\label{s-reflection}

Our final task before describing the reduction algorithm is understanding the action of a simple reflection $s_\alpha$ ($\alpha \in \Delta$) on $X_G$ in terms of the co-ordinates coming from the isomorphism $X_G \cong A \times N$. 
For the remainder of the subsection, we fix $\al \in \Delta$ and a choice of basis vector $X_\alpha \in \frg_\alpha$. This determines the homomorphism $\phi_\alpha : \SL_2 \to \bG$ (as defined in \S \ref{s-set-up}) and the simple reflection $s_\alpha = \phi_\alpha\begin{psmallmatrix} 0 & 1 \\ -1 & 0 \end{psmallmatrix} \in \bG(\bZ)$. We observe that $\phi_\alpha$ intertwines the Cartan involution $x \mapsto {}^t x^{-1}$ of $\SL_2$ with $\theta_0$. In particular, $\phi_\alpha(\SO_2(\bR)) \leq K$. 
\begin{lemma}\label{lem_action_of_simple_reflection}
\begin{enumerate} 
\item We have $s_\alpha \in K \cap \bG(\bZ)$.
\item 
The element $s_\alpha$ normalises $N_\alpha = \ker( p_\alpha : N \to \bR )$. 
\end{enumerate}
\end{lemma}
\begin{proof}
The first part follows since $\begin{psmallmatrix} 0 & 1 \\ -1 & 0 \end{psmallmatrix} \in \SO_2(\bR)$ and $\phi_\al(\SO_2(\bR))\subset K$. The second follows from \cite[Lemma 20(b)]{Steinberg}, where we are using the fact that $s_\al$ permutes the set $\Phi^+ - \{\al\}$. 
\end{proof}
\begin{lemma}\label{lem_SL_2}
Let $\alpha \in \Delta$ and $t \in \bR$. Then we have the equality
\[ K \phi_\alpha\left( \begin{array}{cc} 1 & 0 \\ t & 1 \end{array} \right) = K \phi_\alpha\left( \left( \begin{array}{cc} y & 0 \\ 0 & y^{-1} \end{array} \right) \left( \begin{array}{cc} 1 & t y^{-2}  \\ 0& 1 \end{array} \right) \right) \]
of right cosets of $K$, where $y = \sqrt{1 + t^2}$.
\end{lemma}
\begin{proof}
We can reduce immediately to the case $\bG = \SL_2$, $\phi_\alpha = \operatorname{id}$, in which case this is a matrix calculation.
\end{proof}

\begin{proposition}\label{prop_effect_of_simple_reflection}
Let $x = K a n \in X_G$, let $t = p_\alpha(n) \in \bR$, and let $m = u_\alpha(-t) n \in N_\al$.  Then $xs_\alpha = K a' n'$, where $a' \in A$, $n' \in N$ are given by the following formulae:
\[ a' = \check{\alpha}\left(\sqrt{1 + t^2 \alpha(a)^2}\right) s_\alpha(a), \]
\[ n' = u_{\alpha}\left( \frac{- t \alpha(a)^2}{1 + t^2 \alpha(a)^2 } \right) s_\alpha^{-1} m s_\alpha. \]
\end{proposition}
\begin{proof}
We compute. Since $s_\alpha \in K$, we have $K  a n s_\alpha = K s_\alpha^{-1} a s_\alpha s_\alpha^{-1} n s_\alpha = K s_\alpha(a) s_\alpha^{-1} n s_\alpha$. Writing $n = u_{\alpha}(t) m$, we have 
\[ s_\alpha^{-1} n s_\alpha = \phi_\alpha\left( \begin{array}{cc} 1 & 0 \\ -t & 1 \end{array} \right) s_\alpha^{-1} m s_\alpha. \]
Note that $m \in N_\al = \ker( p_\alpha : N \to \bR )$, so $s_\alpha^{-1} m s_\alpha \in N_\alpha \leq N$, by Lemma \ref{lem_action_of_simple_reflection}. 

We have $s_\alpha(a) = a \check{\alpha}(\alpha(a))^{-1}$. Since $(\al \circ \check\al)(z) = z^2$ for all $z \in \bR^\times$, this implies $\alpha(s_\alpha(a)) = \alpha(a)^{-1}$. Thus
\[ K a n s_\alpha = K s_\alpha(a)  \phi_\alpha\left( \begin{array}{cc} 1 & 0 \\ -t & 1 \end{array} \right) s_\alpha^{-1} m s_\alpha = K  \phi_\alpha\left( \begin{array}{cc} 1 & 0 \\ - \alpha(a) t & 1 \end{array} \right) s_\alpha(a) s_\alpha^{-1} m s_\alpha. \]
(Here we're using the fact that $\phi_\al\begin{psmallmatrix} 1 & 0\\-t & 1 \end{psmallmatrix}$ is in the root group corresponding to $-\al$, so the action of $T$ by conjugation is as described in \S \ref{s-set-up}.)
By Lemma \ref{lem_SL_2}, we have $K  \phi_\alpha\begin{psmallmatrix} 1 & 0 \\ - \alpha(a) t & 1 \end{psmallmatrix} = K \check{\alpha}(y) u_{\alpha}( - \alpha(a) t y^{-2} ),$
where $y = \sqrt{1 + \alpha(a)^2 t^2}$, hence
\begin{multline*}  K a n s_\alpha = K \check{\alpha}(y) u_{\alpha}( - \alpha(a) t y^{-2} ) s_\alpha(a) s_\alpha^{-1} m s_\alpha \\
= K \check{\alpha}(y) s_\alpha(a) u_{\alpha}( - \alpha(s_\alpha(a))^{-1} \alpha(a) t y^{-2} ) s_\alpha^{-1} m s_\alpha 
 \\= K \check{\alpha}(y) s_\alpha(a) u_{\alpha} \left( \frac{- t \alpha(a)^2}{1 + t ^2 \alpha(a)^2} \right) s_\alpha^{-1} m s_\alpha, 
\end{multline*}
which is the desired result. 
\end{proof}

\begin{corollary}\label{cor_movement_towards_fundamental_polygon}
Let $x = Kan \in X_G$. Define $t$ and $a'$ as in the previous proposition, and assume $t \in [-1/2, 1/2]$. Suppose $\delta \in (1/4, 1)$ is such that $\alpha(a)^2 > (\delta - t^2)^{-1}$ (in other words, the $\bG$-Lov\'{a}sz condition fails here). Then we have $\alpha(a') < \delta \alpha(a)$ and $\sigma(a') < \delta \sigma(a)$, where $\sigma = \sum_{\beta \in \Phi^+} \beta \in X^\ast(\bT)$.
\end{corollary}
\begin{proof}
We have $\alpha(a') = (1 + \alpha(a)^2 t^2) \alpha(a)^{-1}$, hence $\alpha(a') / \alpha(a) = \alpha(a)^{-2} + t^2$. We also have $\langle \sigma, \check{\alpha} \rangle = 2$, hence
\[ \sigma(a') = \sigma(\check{\alpha}( \sqrt{1 + t^2 \alpha(a)^2})) \sigma(a \check{\alpha}(\alpha(a))^{-1}) = \sigma(a) (\alpha(a)^{-2} + t^2). \]
Re-arranging now shows that the three conditions $\alpha(a') / \alpha(a)  < \delta$, $\alpha(a) > (\delta - t^2)^{-1/2}$, and $\sigma(a') < \delta \sigma(a)$ are all equivalent. 
\end{proof}

\subsection{The algorithm}\label{s-algorithm}

We maintain the notation of \S \ref{s-reduced}, so $\bG$ is a split semisimple group over $\bZ$ with split maximal torus $\bT$, $\Delta$ is a fixed set of simple roots, $\bN$ is the corresponding maximal unipotent subgroup of $\bG$, and $K \leq G$ is the maximal compact subgroup defined in \S \ref{s-set-up}. In addition, we fix a labelling $\Delta = \{ \alpha_1, \dots, \alpha_r \}$ for the simple roots and, for each $\al \in \Delta$, we fix a basis vector $X_\alpha$ for $\frg_\alpha$ with corresponding root $\SL_2$ given by $\phi_\alpha : \SL_2 \to \bG$. Algorithm \ref{alg_LLL_for_G} describes the main construction of this paper.
\begingroup
\begin{algorithm}
\caption{Lattice reduction for $\mathbf{G}$.} 
\begin{flushleft}
\textbf{Input:} A constant $\delta \in (1/4, 1)$, an allowable fundamental set $\omega \subset N$, and elements $a \in A$, $n \in N$.
\textbf{Output:} An element $\gamma \in \bG(\bZ)$ such that $x = K a n \gamma \in X_G$ is $(\del, \omega)$-reduced, and elements $a(x) \in A$, $n(x) \in N$ such that $x = K a(x) n(x)$. 
\end{flushleft}
  \begin{algorithmic}[1]\label{alg_LLL_for_G}
  \STATE Set $\gamma \gets 1 \in \bG(\bZ)$, $a' \gets a$, $n' \gets n$.
  \STATE\label{alg_size_reduce} Find $\gamma_{n'} \in \bN(\bZ)$ such that $n' \gamma_{n'} \in \omega$.  Set $\gamma \leftarrow \gamma \gamma_{n'}$, $n' \leftarrow n' \gamma_{n'}$. 
\FOR{$i = 1$ \TO $r$}
 \IF{$\alpha_i(a')^2 > (\delta - p_{\alpha_i}(n')^2)^{-1}$}
 \STATE Set $t \gets p_{\alpha_i}(n)$, $m \gets u_{\alpha_i}(-t) n$,  $\gamma \gets \gamma s_{\alpha_i},$ $a' \gets \check{\alpha_i} \left(\sqrt{1 + t^2 \alpha_i(a')^2}\right) s_{\alpha_i}(a'),$
\\ and
$n' \gets u_{\alpha_i}\left( \frac{- t \alpha_i(a')^2}{1 + t^2 \alpha_i(a')^2 } \right) s_{\alpha_i}^{-1} m s_{\alpha_i}$.
\STATE \textbf{go to} 2.
\ENDIF
\ENDFOR
\STATE Set $a(x) \gets a'$, $n(x) \gets n'$.
\RETURN $\gamma, a(x), n(x)$.
\end{algorithmic}
\end{algorithm}
\endgroup

Proposition \ref{prop_effect_of_simple_reflection} shows that at each stage we have $K a n \gamma = K a' n'$, and it is clear that if it the algorithm does terminate, then $K a n \gamma$ is $(\delta, \omega)$-reduced. We need to prove that the algorithm terminates. We will do this by using the geometry of numbers in the Lie algebra $\frg$.
\begin{proof} Let $\kappa$ denote the Killing form of $\frg_\bR$, and let $(\cdot, \cdot)_0$ denote the inner product $(X, Y)_0 = - \kappa(X, d\theta_0 Y)$ on $\frg_\bR$ (it is positive definite, and therefore an inner product, because $\theta_{0, \bR}$ is a Cartan involution.) If $g \in G$, then we write $(\cdot, \cdot)_g$ for the inner product $(X, Y)_g = (\Ad(g)(X), \Ad(g)(Y))_0$. Thus $(\cdot, \cdot)_g$ only depends on the image of $g$ in $X_G = K \backslash G$. We observe that $(\cdot, \cdot)_0$ and $(\cdot, \cdot)_g$ determine inner products on the exterior powers $\bigwedge^r \frg_\bR$ in a natural way. We will also (by abuse of notation) denote these by $(\cdot, \cdot)_0$ and $(\cdot, \cdot)_g$, respectively. 

Let $\frb = \frakt \oplus \sum_{\al \in \Phi^+} \frakg_\al$, and let $\ell = \dim \frb$. Then $\bigwedge^\ell \frb$ is a rank-one $\bZ$-submodule of $\bigwedge^\ell \frakg$. Choose $b \in \bigwedge^\ell \frb$ to be a basis vector. 
We can compute for any $a_1 \in A$, $n_1 \in N$:
\[ ( b, b)_{a_1n_1} = ( \Ad(a_1 n_1)b, \Ad(a_1 n_1)b)_0 = \sigma(a_1)^2 (b, b)_0, \]
where $\sigma = \sum_{\beta \in \Phi^+} \beta \in X^\ast(\bT)$, as before. 

Suppose then that we begin with $a \in A$, $n \in N$, and that the pair $(a_1, n_1)$ is the result of carrying out Step 3 of the above algorithm for the $k^\text{th}$ time, for some $k \geq 1$ (so in particular we are given $\gamma_1 \in \bG(\bZ)$ such that $K a_1 n_1 = K a n \gamma_1$). By Corollary \ref{cor_movement_towards_fundamental_polygon}, we have $\sigma(a_1) < \delta^k \sigma(a)$, hence
\[ ( \Ad(\gamma_1) b, \Ad(\gamma_1) b)_{an} = ( b, b )_{an\gamma_1} = (b, b)_{a_1 n_1} < \delta^{2k} (b, b)_{an}. \]
We now observe that $\Ad(\gamma_1) b$ is another non-zero vector in the lattice $\bigwedge^\ell \frg$. Let $c_1 = (b, b)_{an}^{1/2}$ denote the length of $b$ with respect to the inner product $(\cdot, \cdot)_{an}$, and let $c_0$ denote the minimal length of a non-zero vector in $\bigwedge^\ell \frg$, with respect to the inner product $(\cdot, \cdot)_{an}$. We finally obtain the inequality $c_0 < \delta^k c_1$, hence $k < \log(c_0 / c_1) / \log(\delta) = \log(c_1 / c_0) / \log(\delta^{-1})$. The right-hand side here only depends on $an$ and $\delta$, showing that the algorithm must terminate after at most $\lfloor  \log(c_1 / c_0) / \log(\delta^{-1}) \rfloor$ iterations of Step 3. 
\end{proof}
When $\bG = \SL_g$, the quantity $\sigma(a)$ appearing in the above proof equals the quantity $D$, defined on \cite[p. 521]{Len82} and used in the proof of termination of the LLL algorithm (as can be seen using Lemma \ref{lem_Gram_Schmidt_for_SL_g}, and noting that working in $\SL_g$ means we restrict to lattices of covolume 1). The proof of termination in this case therefore reduces to the one given in \cite{Len82}. 

 In \S\S \ref{sec_case_Sp_2g} -- \ref{sec_case_G_2}, we will make Algorithm \ref{alg_LLL_for_G} more explicit  for the groups  $\bG = \Sp_{2g}$, $\SO_{2g}$, and $G_2$. We make some general remarks relating to implementation:
\begin{itemize}
\item For Step 2, we need to specify $\omega$ and give a procedure that, given $n \in N$, finds $\gamma_n \in \bN(\bZ)$ such that $n \gamma_n \in \omega$. If we fix an ordering of the positive roots, together with a basis vector $X_\alpha$ for each $\alpha \in \Phi^+$, then the product map $\prod_{\alpha \in \Phi^+}  u_\alpha : \bR^{\Phi^+} \to N$ is a diffeomorphism (see \cite[Lemma 18]{Steinberg}). We can use these co-ordinates, together with the commutation relations in the nilpotent group $N$ (cf. \cite[Lemma 15]{Steinberg}), to describe an explicit choice for $\omega$ together with a method of size reduction, `uniformly in the group $\bG$'. We give details in the appendix.

However, in the concrete cases we consider below (including for the case of the exceptional group $G_2$), we will find that the representation of group elements as matrices leads to a more natural choice of $\omega$ and corresponding size-reduction procedure that is closer in spirit to the one used in \cite{Len82}. 

\item In cases of interest, the symmetric space $X_G$ is often identified with a more concrete space (such as the space of covolume-one inner products when $\bG = \SL_g)$. It is therefore desirable to have an explicit way of computing the corresponding co-ordinates $(a, n) \in A \times N$ before proceeding to Step 1. In each of the cases we consider below, this will be done using the Gram--Schmidt process (again, even including for the case of the exceptional group $G_2$).

\item A very important feature of the LLL algorithm is that it terminates in polynomial time for integer lattices, and in general the number of steps required can be bounded explicitly in terms of $\min_j (e_j^\ast, e_j^\ast)_H$ (see \cite[Proposition 1.26]{Len82} or \cite[Theorem 1.1]{Val94}). The above proof that the algorithm terminates implies that the same is true in general. 
\end{itemize} 

\section{The case $\bG = \Sp_{2g}$}\label{sec_case_Sp_2g}

We now make Algorithm \ref{alg_LLL_for_G} more explicit in the case of the symplectic group. We first introduce notation. More details about the structure of $\bG$ (including the root system) can be found, for example, in \cite[\S 16]{FultonHarris}. 

Let $g \geq 2$ be an integer, and let $e_{-g}, e_{1-g}, \dots, e_{-1}, e_1, \dots, e_g$ denote the standard basis of $\bR^{2g}$ (with a shift in indices). Let $\Psi$ denote the $g \times g$ matrix with $1$'s on the antidiagonal and $0$'s elsewhere, and let
\[ S = \left( \begin{array}{cc} 0 & \Psi \\ - \Psi & 0 \end{array}\right). \]
We define
\[ \Sp_{2g} = \{ x \in \GL_{2g} \mid {}^t x S x = S \}. \]
This formula defines a split semisimple group over $\bZ$. A split maximal torus is 
\[ \bT = \{ \diag(t_{-g}, \dots, t_{-1}, t_{-1}^{-1}, \dots, t_{-g}^{-1}) \} \leq \Sp_{2g}. \]
If $-g \leq i \leq g$ and $i \neq 0$, let $\varepsilon_i \in X^\ast(\bT)$ be the character sending $\diag(t_{-g}, \dots, t_g)$ to $t_i$. The roots $\Phi = \Phi(\Sp_{2g}, \bT)$ are given by $\varepsilon_i - \varepsilon_j$ ($-g \leq i, j \leq -1$, $i \neq j$) and $\pm(\varepsilon_i + \varepsilon_j)$ ($-g \leq i \leq j \leq -1$). A system of positive roots is $\Phi^+ = \{ \varepsilon_i - \varepsilon_j \mid -g \leq i < j \leq -1 \} \cup \{ \varepsilon_i + \varepsilon_j \mid -g \leq i \leq j \leq -1 \}$. The corresponding set of simple roots is $\Delta = \{\varepsilon_i - \varepsilon_{i + 1} \mid -g \leq i \leq -2\} \cup \{2\varep_{-1}\}$. The corresponding Borel subgroup $\bB \subset \Sp_{2g}$ is the subgroup of upper-triangular matrices preserving $S$; its unipotent radical $\bN$ is the subgroup of unipotent upper-triangular matrices preserving $S$. The group $A = \bT(\bR)^\circ$ is the subgroup of diagonal matrices with positive entries preserving $S$. 

The Cartan involution $\theta_0 : \Sp_{2g} \to \Sp_{2g}$ is given by the formula $\theta_0(x) = {}^t x^{-1}$. We write $K \leq \Sp_{2g}(\bR)$ for its fixed-point subgroup, a maximal compact subgroup of $\Sp_{2g}(\bR)$. The associated symmetric space $X_{\Sp_{2g}(\bR)}$ is the quotient $K \backslash \Sp_{2g}(\bR)$. We now describe alternative realisations of the symmetric space:
\begin{proposition}\label{prop_interpretation_of_symmetric_space}
There are $\Sp_{2g}(\bR)$-equivariant bijections between the following sets:
\begin{enumerate}
\item The symmetric space $X_{\Sp_{2g}(\bR)}$.
\item The set of inner products $H$ on $\bR^{2g}$ (equivalently, the set of positive-definite symmetric matrices in $M_{2g}(\bR)$) that are compatible with $S$, in the sense that $H^{-1}( {}^t S) = S^{-1} ({}^t H)$.
\item The set of $J \in \Sp_{2g}(\bR)$ such that $J^2 = -1$ and ${}^tJ S$ is positive definite. 
\item The set of matrices $\Pi  \in M_{g}(\bC)$ such that $\Pi$ is symmetric and $\operatorname{Im}( \Pi )$ is positive definite. 
\end{enumerate}
\end{proposition}
The fourth set is the standard realisation of the Siegel upper half-space.
\begin{proof}
We first describe the (right) actions of $\Sp_{2g}(\bR)$ on the sets (1) -- (3). The action on $X_{\Sp_{2g}(\bR)}$ is the usual one on right cosets. If $x \in \Sp_{2g}(\bR)$ and $H$ is an inner product as in (2), then we define $ H \cdot x = {}^t x H x$. If $J$ is as in (3), we define $J \cdot x = x^{-1} J x$. 

Let $H_0$ denote the standard inner product on $\bR^{2g}$.
Given $x \in \Sp_{2g}(\bR)$, we associate the inner product ${}^t x H_0 x$. The stabiliser equals $K$, so this gives an injective, $\Sp_{2g}(\bR)$-equivariant map $(1) \rightarrow (2)$. The argument of \cite[Lemma 7.11]{Grayson} shows that $\Sp_{2g}(\bR)$ acts transitively on the set of inner products in (2). Thus the map is bijective. 
Given $H$ in $(2)$, let $J = - S^{-1} H$. An elementary computation shows that $J$ lies in (3), and that the map $H \mapsto J$ is $\Sp_{2g}(\bR)$-equivariant. An inverse is given by $H = - SJ$. Thus (2) and (3) are in bijection.

The bijection $(3) \to (4)$ is part of the `linear algebra of polarised abelian varieties', and is described in e.g. \cite[Ch. II, \S 4, Lemma 4.1]{Mum07} or \cite[Theorem 4.2.1]{Bir04}. We will not use (4) here so we omit the details. 
\end{proof}
By definition, $X_{\Sp_{2g}(\bR)}$ comes equipped with the base point $K$. Under the above bijections, this point corresponds to $H = H_0$ (the standard inner product on $\bR^{2g}$) in (2), $J = S$ in (3), and $\Pi_0 = i 1_g$ in (4). 

Recall that the Iwasawa decomposition gives us a bjiection
\[ A \times N \to X_{\Sp_{2g}(\bR)}, (a, n) \mapsto K a n. \]
It follows that any inner product $H$ compatible with $S$ can be expressed as ${}^t (an) an$ for a unique pair $(a, n) \in A \times N$. This pair can be computed from $H$ using the usual Gram--Schmidt orthogonalisation process. Let $e^\ast_{-g}, \dots, e^\ast_g$ denote the result of carrying out this process on the standard basis $e_{-g}, \dots, e_g$, using the inner product $H$. Then we have the formulae
\begin{eqnarray*}
e_{-g} &=& e^\ast_{-g},\\ 
&\vdots&\\ 
e_j &=& e_j^\ast + \sum_{\substack{i = -g\\i\neq 0}}^{j-1} \mu_{i, j} e_i^\ast\\
& \vdots&\\
 e_g &=& e_g^\ast + \sum_{\substack{i = -g \\ i \neq 0}}^{g-1} \mu_{i, g} e_i^\ast, 
 \end{eqnarray*}
with $ \mu_{i, j} = ( e_j, e_i^\ast)_H / (e_i^\ast, e_i^\ast)_H$.
In particular, we take $\mu_{i, i} = 1$ and $\mu_{i, j} = 0$ if $i > j$. 
\begin{lemma}\label{lem_Gram_Schmidt}
Given an inner product $H$ compatible with with $S$, let $n = n_H = ( \mu_{i, j} )_{-g \leq i, j \leq g, i j \neq 0}$ and $a = a_H = \diag( \| e_i^\ast \|_H)_{-g \leq i \leq g, i \neq 0}$ as defined above. Then $a \in A$, $n \in N$, and $H = {}^t(an) an$. 
\end{lemma}
\begin{proof}
This follows from Lemma \ref{lem_Gram_Schmidt_for_SL_g}, since the groups $A$, $N \leq \Sp_{2g}(\bR)$ are contained in the corresponding groups for $\SL_{2g}$ (which is why we chose the embedding $\Sp_{2g} \leq \SL_{2g}$ that we did). 
\end{proof}
We now discuss size reduction. We define $\omega \subset N$ to be the set of matrices $n$ such that $|n_{i, j}| \leq 1/2$ for each pair $(i, j)$ with either $-g \leq i < j \leq -1$ or $-g \leq i \leq -j \leq -1$. This set is compact as any element of $N$ is determined by these matrix entries. To show that $\omega$ is an allowable fundamental set, we describe the homomorphisms $u_\alpha : \bG_a \to \Sp_{2g}$ associated to $\alpha \in \Phi^+$.
 \begin{itemize} \item If $-g \leq i < j \leq -1$, and $\alpha = \varepsilon_i - \varepsilon_j$, then $u_\alpha(t) = 1 + t (E_{i, j} - E_{-i, -j})$ (where $E_{i, j}$ is the associated elementary matrix).  Right multiplication by $u_{\alpha}(t)$ corresponds to adding $t$ times column $i$ to column $j$ and $t$ times column $-i$ to column $-j$. 
 \item If $-g \leq i \leq -1$, and $\alpha = 2 \varepsilon_i$, then $u_\alpha(t) = 1 + t E_{i, -i}$. Right multiplication by $u_{\alpha}(t)$ corresponds to adding $t$ times column $i$ to column $-i$. 
 \item If $-g \leq i < j \leq -1$, and $\alpha = \varepsilon_i + \varepsilon_j$, then $u_\alpha(t) = 1 + t (E_{i, -j} + E_{j, -i})$. Right multiplication by $u_\alpha(t)$ corresponds to adding $t$ times column $i$ to column $-j$ and $t$ times column $j$ to column $-i$. 
 \end{itemize} 
If $n \in N$, we can produce an element $\gamma_n \in \bN(\bZ)$ such that $n \gamma_n \in \omega$ by successively carrying out column operations as follows: 
\begin{itemize}
\item For each $j = -g + 1, -g + 2, \dots, -2, -1$, successively make the entries $n_{i, j}$ (for $i = j-1, j - 2, \dots, -g + 1, -g$) lie in $[-1/2, 1/2)$ by multiplying on the right by $u_{\varep_i - \varep_j}(m_{i,j})$ for some integer $m_{i, j}$. 
\item For each $j = 1, 2, \dots, g$, successively make the entries $n_{i, j}$ (for $i = -j, -j - 1, \dots, -g$) lie in $[-1/2, 1/2)$ by multiplying on the right by $u_{\varep_i + \varep_{-j}}(m_{i, j})$ for some integer $m_{i, j}$. 
\end{itemize}
We observe that the order of these operations is such that, once we have forced a given matrix entry to lie in $[-1/2, 1/2]$, it is not disturbed by any later column operations. Finally we observe that $\omega$ is an allowable fundamental set because the homomorphism $p : N \to \bR^\Delta$ takes a matrix $n$ to its tuple of entries $(n_{-g, 1-g}, n_{1-g, 2-g}, \dots, n_{-2, -1}, n_{-1, 1})$.

We next make explicit the simple reflections $s_{-g}, \dots, s_{-1}$ respectively corresponding to the chosen simple roots $\varepsilon_i - \varepsilon_{i+1}$ ($i = -g, \dots, -2$), $2 \varepsilon_{-1}$, by describing their action on basis vectors: 
\begin{itemize}
\item For each $i = -g, \dots, -2$, $s_i$ sends $e_i$ to $-e_{i+1}$, $e_{i+1}$ to $e_i$, $e_{-i}$ to $e_{-(i+1)}$, $e_{-(i+1)}$ to $-e_i$, and fixes the remaining basis vectors.
\item $s_{-1}$ sends $e_{-1}$ to $-e_1$, $e_1$ to $e_{-1}$, and fixes the remaining basis vectors. 
\end{itemize}

We are now ready to give a more explicit version of Algorithm \ref{alg_LLL_for_G}.
\begingroup
\begin{algorithm}
\caption{Lattice reduction for $\Sp_{2g}$.} 
\begin{flushleft}
\textbf{Input:} A constant $\delta \in (1/4, 1)$, and a positive definite $2g \times 2g$ real symmetric matrix $H$ compatible with $S$.

\textbf{Output:} An element $\gamma \in \Sp_{2g}(\bZ)$ such that ${}^t \gamma H \gamma$ corresponds to a $(\delta, \omega)$-reduced point of $X_{\Sp_{2g}(\bR)}$ under the bijection of Proposition \ref{prop_interpretation_of_symmetric_space}.
\end{flushleft}
  \begin{algorithmic}[1]\label{alg_LLL_for_Sp_2g}
  \STATE Using Gram--Schmidt, compute $a \in A, n \in N$ such that $H = {}^t (an) an$. Set $\gamma \gets 1 \in \Sp_{2g}(\bZ)$, $a' \gets a$, $n' \gets n$, $k \gets -g$.
\WHILE{$k < - 1$}
\STATE Execute RED($k, k+1$).
\IF{ $\alpha_k(a')^2 > (\delta - (n'_{k, k+1})^2)^{-1}$ }
\STATE Execute REFL($k$).
\IF{$k > -g$}
\STATE $k \gets k-1$.
\ENDIF
\ELSE
\FOR{ $i = k-1$, $i \geq -g$, $i \gets i-1$ }
\STATE Execute RED($i, k+1$).
\ENDFOR 
\STATE Set $k \leftarrow k+1$.
\ENDIF
\ENDWHILE
\STATE Execute RED($-1, 1$). 
\IF{$\alpha_{-1}(a')^2 > (\delta - (n'_{-1, 1})^2)^{-1}$}
\STATE Execute REFL($-1$). Set $k \leftarrow k-1$.
\STATE \textbf{go to} 2.
\ELSE
\FOR{$i = -2$, $i \geq -g$, $i \gets i-1$}
\STATE Execute RED($i, 1$).
\ENDFOR
\FOR{$j = 2$, $j \leq g$, $j \gets j+1$}
\FOR{$i = -j$, $i \geq -g$, $i \gets i-1$}
\STATE Execute RED($i, j$).
\ENDFOR
\ENDFOR
\ENDIF
\RETURN $\gamma$.
\end{algorithmic}
\end{algorithm}
\begin{algorithm}
\caption{Subroutine RED($i, j$) for Algorithm \ref{alg_LLL_for_Sp_2g}.}
\begin{algorithmic}
\IF{$-g \leq i < j \leq -1$}
\STATE Let $m$ be the unique integer such that $n'_{i, j} + m \in [-1/2, 1/2)$. Set $n' \leftarrow n u_{\varepsilon_i - \varepsilon_j}(m)$, $\gamma \leftarrow \gamma u_{\varepsilon_i - \varepsilon_j}(m)$. 
\ELSIF{$-g \leq i \leq -j \leq -1$}
\STATE Let $m$ be the unique integer such that $n'_{i, j} + m \in [-1/2, 1/2)$. Set $n' \leftarrow n u_{\varepsilon_i + \varepsilon_{-j}}(m)$, $\gamma \leftarrow \gamma u_{\varepsilon_i + \varepsilon_{-j}}(m)$.
\ENDIF
\end{algorithmic}
\end{algorithm}
\begin{algorithm}
\caption{Subroutine REFL($i$) for Algorithm \ref{alg_LLL_for_Sp_2g}.}
\begin{algorithmic}
\IF{$-g \leq i \leq -2$}
\STATE Set $t \leftarrow n_{i, i+1}$, $m' \leftarrow u_{\varepsilon_{i} - \varepsilon_{i+1}}( t )^{-1} n'$, $n' \leftarrow u_{\varepsilon_i - \varepsilon_{i+1}}( (-t \alpha_i(a')^2 )/(1 + t^2 \alpha_i(a')^2) )s_i^{-1} m' s_i$, $a' \leftarrow \check{\alpha}_i(\sqrt{1 + t^2 \alpha_i(a')^2}) s_i(a')$, $\gamma \leftarrow \gamma s_i$.
\ELSIF{$i = -1$}
\STATE Set $t \leftarrow n_{-1, 1}$, $m' \leftarrow u_{2 \varepsilon_i}(t)^{-1} n'$, $n' \leftarrow u_{2 \varepsilon_i}( (-t \alpha_i(a')^2 )/(1 + t^2 \alpha_i(a')^2) )s_i^{-1} m' s_i$, $a' \leftarrow \check{\alpha}_i(\sqrt{1 + t^2 \alpha_i(a')^2}) s_i(a')$, $\gamma \leftarrow \gamma s_i$.  
\ENDIF
\end{algorithmic}
\end{algorithm}
\endgroup

Looking at the steps, we can we first carry out the algorithm of \cite{Len82} for the top left $g \times g$-submatrix of the element $an \in \Sp_{2g}(\bR)$, eventually reaching $k = -1$, where we have the possibility of acting by the simple reflection $s_{-1}$. This can be seen in terms of the Dynkin diagram of $\Sp_{2g}$: 
\begin{center}
\begin{tikzpicture}[transform shape, scale=.8]
\node[root] (a) {}; 
\node[root] (b) [right=of a] {};
\node[root] (c) [right=of b] {};
\node[root] (d) [right=of c] {};
\node[root] (e) [right=of d] {};
\node[root] (f) [right=of e] {};

\draw[thick] (a) -- (b) -- (c) ;
\draw[thick, dashed] (c) -- (d) ;
\draw[thick] (d) -- (e);
\draw[postaction={decorate}, 
  decoration={markings,mark=at position .8
  with {\arrow{angle 60}}}, thick, double distance=2pt] (e) -- (f);
\end{tikzpicture} 
\end{center} 
The group $\SL_g$ may be embedded in $\Sp_{2g}$ as a Levi subgroup (more precisely, as the Levi subgroup of block diagonal matrices, with blocks of size $g$), and its Dynkin diagram may be identified with the subdiagram of $\Sp_{2g}$ obtained by deleting the right-most vertex, which corresponds to the simple root $2 \varepsilon_{-1}$. From this point of view, we see that Algorithm \ref{alg_LLL_for_Sp_2g} is essentially the one described in \cite{Dec04}, modulo the translation between our interpretation of $X_{\Sp_{2g}}$ and the identification with the Siegel upper half-space. The choice of ordering of the simple reflections is arbitrary; if we take the reverse ordering, then we essentially recover \cite[Algorithm 4]{Gam06}. 

\section{The case $\bG = \SO_{2g}$}\label{sec_case_SO_2g}

We next make Algorithm \ref{alg_LLL_for_G} explicit in the case of an even orthogonal group. (The case of a split odd orthogonal group is rather similar to the case considered in \cite{Thorne}.)   Details about the root system can be found, for example, in \cite[\S 18]{FultonHarris}.

Let $g \geq 3$ be an integer, and let $e_{-g}, \dots, e_{-1}, e_1, \dots, e_g$ denote the standard basis of $\bR^{2g}$. Let $\Psi$ denote the $g \times g$ matrix with $1$'s on the antidiagonal and $0$'s elsewhere, and let
\begin{equation}\label{eqn_matrix_of_symmetric_form} S = \left( \begin{array}{cc} 0 & \Psi \\ \Psi & 0 \end{array}\right). 
\end{equation}
Then $S$ defines a bilinear form whose associated quadratic form is 
\[ q(v_{-g}, \dots, v_g) = \sum_{i=1}^g v_{-i} v_i. \]
We define
\[ \SO_{2g} = \{ x \in \SL_{2g} \mid q(xv) = q(v) \}. \]
This formula defines a split semisimple group over $\bZ$ (cf. \cite[Appendix C]{Con14}). A split maximal torus is 
\[ \bT = \{ \diag(t_{-g}, \dots, t_{-1}, t_{-1}^{-1}, \dots, t_{-g}^{-1}) \} \leq \SO_{2g}. \]
If $-g \leq i \leq g$ and $i \neq 0$, let $\varepsilon_i \in X^\ast(\bT)$ be the character sending $\diag(t_{-g}, \dots, t_g)$ to $t_i$. The roots $\Phi = \Phi(\SO_{2g}, \bT)$ are given by $\varepsilon_i - \varepsilon_j$ ($-g \leq i, j \leq -1$, $i \neq j$) and $\pm(\varepsilon_i + \varepsilon_j)$ ($-g \leq i < j \leq -1$). A root basis is given by $\Delta = \{ \varepsilon_i - \varepsilon_{i+1} \mid i = -g, \dots, -2 \} \sqcup \{ \varepsilon_{-1} + \varepsilon_{-2} \}$, with corresponding set of positive roots $\Phi^+ = \{ \varepsilon_i - \varepsilon_j \mid -g \leq i < j \leq -1 \} \cup \{ \varepsilon_i + \varepsilon_j \mid -g \leq i <j \leq -1 \}$. The corresponding subgroup $\bN \leq \SO_{2g}$ is the subgroup of unipotent upper-triangular matrices preserving $S$. The group $A = \bT(\bR)^\circ$ is the subgroup of diagonal matrices with positive entries preserving $S$. 

The Cartan involution $\theta_0 : \SO_{2g} \to \SO_{2g}$ is again given by the formula $\theta_0(x) = {}^t x^{-1}$. We write $K \leq \SO_{2g}(\bR)$ for its fixed-point subgroup, a maximal compact subgroup of $\SO_{2g}(\bR)$.
\begin{proposition}\label{prop_interpretation_of_symmetric_space_case_SO_2g}
There are $\SO_{2g}(\bR)$-equivariant bijections between the following sets:
\begin{enumerate}
\item The symmetric space $X_{\SO_{2g}(\bR)}$.
\item The set of inner products $H$ on $\bR^{2g}$ (equivalently, positive definite symmetric matrices in $M_{2g}(\bR)$) that are compatible with $S$, in the sense that $H^{-1} ({}^t S) = S^{-1} ({}^t H)$.
\end{enumerate}
\end{proposition}
\begin{proof}
This may be proved in the same way as \cite[Proposition 3.1]{Thorne}.
\end{proof}
The Iwasawa decomposition gives us a bjiection
\[ A \times N \to X_{\SO_{2g}(\bR)}, (a, n) \mapsto K a n. \]
It follows that any inner product $H$ compatible with $S$ can be expressed as ${}^t (an) an$ for a unique pair $(a, n) \in A \times N$. Just as in the case of the symplectic group, this pair may be computed using the Gram--Schmidt process.  Let $e^\ast_{-g}, \dots, e^\ast_g$ denote the result of carrying out this process on the standard basis $e_{-g}, \dots, e_g$, using the inner product $H$, and let $\mu_{i, j} = (e_j, e_i^\ast)_H/(e_i^\ast, e_i^\ast)_H$. 
\begin{lemma}\label{lem_Gram_Schmidt_SO}
Given an inner product $H$ compatible with $S$, let $n = n_H = ( \mu_{i, j} )_{-g \leq i, j \leq g, i j \neq 0}$ and $a = a_H = \diag( \| e_i^\ast \|_H)_{-g \leq i \leq g, i \neq 0}$ as defined above. Then $a \in A$, $n \in N$, and $H = {}^t(an) an$. 
\end{lemma}
\begin{proof}
The proof is the same as for Lemma \ref{lem_Gram_Schmidt}. 
\end{proof}
We now discuss size reduction. We define $\omega \subset N$ to be the set of matrices $n$ such that $|n_{i, j}| \leq 1/2$ for each pair $(i, j)$ with either $-g \leq i < j \leq -1$ or $-g \leq i < -j \leq -1$. This set is compact as any element of $N$ is determined by these matrix entries. 
\begin{proposition}
The set $\omega$ is an allowable fundamental set. 
\end{proposition}
\begin{proof}
Let $n \in N$. We construct $\gamma_n$ such that $n\gamma_n \in \omega$ by successively performing column operations corresponding to elements $u_\alpha(m) \in \bN(\bZ)$ for $\alpha \in \Phi^+$. More explicitly:
\begin{itemize}
\item For each $j = -g + 1, \dots, -2, -1$, successively make the entries $n_{i, j}$ (for $i = j-1, j - 2, \dots, -g$) lie in $[-1/2, 1/2)$ by multiplying on the right by $u_{\varepsilon_i - \varepsilon_j}(m_{i, j})$ for some integer $m_{i, j}$. 
\item For each $j = 1, 2, \dots, g$, successively make the entries $n_{i, j}$ (for $i = -j - 1, -j - 2, \dots, -g$) lie in $[-1/2, 1/2)$ by multiplying on the right by $u_{\varepsilon_i + \varepsilon_{-j}}(m_{i, j})$ for some integer $m_{i, j}$. 
\end{itemize}
We observe that the order of these operations is such that, once we have forced a given matrix entry to lie in $[-1/2, 1/2)$, it is not disturbed by any later column operations. 
 \end{proof}
We next make explicit the simple reflections $s_{-g}, \dots, s_{-1}$ respectively corresponding to the chosen simple roots $\varepsilon_i - \varepsilon_{i+1}$ ($i = -g, \dots, -2$), $\varepsilon_{-1} + \varepsilon_{-2}$, by describing their action on basis vectors: 
\begin{itemize}
\item For each $i = -g, \dots, -2$, $s_i$ sends $e_i$ to $-e_{i+1}$, $e_{i+1}$ to $e_i$, $e_{-i}$ to $e_{-(i+1)}$, $e_{-(i+1)}$ to $-e_i$, and fixes the remaining basis vectors.
\item $s_{-1}$ sends $e_{-1}$ to $-e_2$, $e_{-2}$ to $e_1$, $e_1$ to $-e_{-2}$, $e_2$ to $e_{-1}$, and fixes the remaining basis vectors. 
\end{itemize}
We are now ready to give a more explicit version of Algorithm \ref{alg_LLL_for_G} in this case.
\begingroup
\begin{algorithm}
\caption{Lattice reduction for $\SO_{2g}$.} 
\begin{flushleft}
\textbf{Input:} A constant $\delta \in (1/4, 1)$, and a positive definite $2g \times 2g$ real symmetric matrix $H$ compatible with $S$.

\textbf{Output:} An element $\gamma \in \SO_{2g}(\bZ)$ such that ${}^t \gamma H \gamma$ corresponds to a $(\delta, \omega)$-reduced point of $X_{\SO_{2g}(\bR)}$ under the bijection of Proposition \ref{prop_interpretation_of_symmetric_space_case_SO_2g}.
\end{flushleft}
  \begin{algorithmic}[1]\label{alg_LLL_for_SO_2g}
  \STATE Using Gram--Schmidt, compute $a \in A, n \in N$ such that $H = {}^t (an) an$. Set $\gamma \gets 1 \in \SO_{2g}(\bZ)$, $a' \gets a$, $n' \gets n$, $k \gets -g$.
\WHILE{$k < - 1$}
\STATE Execute RED($k, k+1$).
\IF{ $\alpha_k(a')^2 > (\delta - (n'_{k, k+1})^2)^{-1}$ }
\STATE Execute REFL($k$).
\IF{$k > -g$}
\STATE $k \gets k-1$.
\ENDIF
\ELSE
\FOR{ $i = k-1$, $i \geq -g$, $i \gets i-1$ }
\STATE Execute RED($i, k+1$).
\ENDFOR 
\STATE Set $k \leftarrow k+1$.
\ENDIF
\ENDWHILE
\STATE Execute RED($-2, 1$). 
\IF{$\alpha_{-1}(a')^2 > (\delta - (n'_{-1, 1})^2)^{-1}$}
\STATE Execute REFL($-1$). Set $k \leftarrow k-2$.
\STATE \textbf{go to} 2.
\ELSE
\FOR{$i = -3$, $i \geq -g$, $i \gets i-1$}
\STATE Execute RED($i, 1$).
\ENDFOR
\FOR{$j = 2$, $j \leq g$, $j \gets j+1$}
\FOR{$i = -j-1$, $i \geq -g $, $i \gets i-1$}
\STATE Execute RED($i, j$).
\ENDFOR
\ENDFOR
\ENDIF
\RETURN $\gamma$.
\end{algorithmic}
\end{algorithm}
\begin{algorithm}
\caption{Subroutine RED($i, j$) for Algorithm \ref{alg_LLL_for_SO_2g}.}
\begin{algorithmic}
\IF{ $-g \leq i < j \leq -1$ }
\STATE Let $m$ be the unique integer such that $n'_{i, j} + m \in [-1/2, 1/2)$. Set $n' \gets n u_{\varepsilon_i - \varepsilon_j}(m)$, $\gamma \leftarrow \gamma u_{\varepsilon_i - \varepsilon_j}(m)$. 
\ELSIF{ $-g \leq i < -j \leq -1$}
\STATE Let $m$ be the unique integer such that $n'_{i, j} + m \in [-1/2, 1/2)$. Set $n' \leftarrow n u_{\varepsilon_i + \varepsilon_{-j}}(m)$, $\gamma \leftarrow \gamma u_{\varepsilon_i + \varepsilon_{-j}}(m)$.
\ENDIF
\end{algorithmic}
\end{algorithm}
\begin{algorithm}
\caption{Subroutine REFL($i$) for Algorithm \ref{alg_LLL_for_SO_2g}.}
\begin{algorithmic}
\IF{$-g \leq i \leq -2$} 
\STATE Set $t \leftarrow n_{i, i+1}$, $m' \leftarrow u_{\varepsilon_{i} - \varepsilon_{i+1}}( t )^{-1} n'$, $n' \leftarrow u_{\varepsilon_i - \varepsilon_{i+1}}( (-t \alpha_i(a')^2 )/(1 + t^2 \alpha_i(a')^2) )s_i^{-1} m' s_i$, $a' \leftarrow \check{\alpha}_i(\sqrt{1 + t^2 \alpha_i(a')^2}) s_i(a')$, $\gamma \leftarrow \gamma s_i$.
\ELSIF{$i = -1$} 
\STATE
Set $t \leftarrow n_{-1, 1}$, $m' \leftarrow u_{2 \varepsilon_i}(t)^{-1} n'$, $n' \leftarrow u_{\varepsilon_{-g+1}+\varepsilon_{-g}}( (-t \alpha_i(a')^2 )/(1 + t^2 \alpha_i(a')^2) )s_i^{-1} m' s_i$, $a' \leftarrow \check{\alpha}_i(\sqrt{1 + t^2 \alpha_i(a')^2}) s_i(a')$, $\gamma \leftarrow \gamma s_i$.  
\ENDIF
\end{algorithmic}
\end{algorithm}
\endgroup

We illustrate the algorithm with a numerical example, similar in nature to the one considered in \cite{Thorne} for the odd orthogonal group. The group $\bG$ acts on the space $V$ of self-adjoint $2g \times 2g$ matrices $T$ of trace 0. In the paper \cite{Sha18}, the $\bG(\bZ)$-orbits of matrices in $V(\bZ)$ with fixed separable characteristic polynomial $f(X) \in \bZ[X]$ are related to elements of the 2-Selmer group of the Jacobian of the hyperelliptic curve $Y^2 = f(X)$. The main construction of \cite{Thorne} gives a $G$-equivariant map $V(\bR)^s \to X_{\SO_{2g}(\bR)}$, $T \mapsto H_T$, where $V(\bR)^s \subset V(\bR)$ is the open subset of real matrices whose characteristic polynomial is separable. Reducing $H_T$ is a method to simplify $T$.

To illustrate, take $g = 4$ and consider the matrix
\[ \Scale[0.7]{ T =  \left(
\begin{array}{cccccccc}
 -14 & -200 & -2808 & -456 & -39095 & 5612 & 11183 & -11932 \\
 1 & 0 & 0 & 0 & 0 & -\frac{117}{2} & 10 & 11183 \\
 0 & 1 & 0 & 2 & 1 & -48 & -\frac{117}{2} & 5612 \\
 0 & 0 & \frac{1}{2} & 14 & 0 & 1 & 0 & -39095 \\
 0 & 0 & 1 & 0 & 14 & 2 & 0 & -456 \\
 0 & 0 & 0 & 1 & \frac{1}{2} & 0 & 0 & -2808 \\
 0 & 0 & 0 & 0 & 0 & 1 & 0 & -200 \\
 0 & 0 & 0 & 0 & 0 & 0 & 1 & -14 \\
\end{array}
\right),} \]
which is associated to the integral point $P = (14, 38867)$ of the genus-3 hyperelliptic curve
\[ Y^2 = X^8 + 4 X^6 + 8 X^5 + 11 X^4 + 4 X^3 + 2 X^2 - 4 X + 9. \]
(More precisely, the proof of \cite[Theorem 10]{Sha18} shows how rational points of the Jacobian of the curve give rise to $\bG(\mathbb{Q})$-orbits; the proof of \cite[Corollary 13]{Sha18} yields, in the simple case of an orbit corresponding to an integral point on the curve, an easy construction of a rational representative for this orbit whose denominators are bounded powers of $2$.) The reduction covariant $H_T \in X_{\SO_8(\bR)}$ may be computed using the same method as in \cite[\S 3]{Thorne}; we obtain the inner product on $\bR^8$, compatible with $S$, given (to 6 significant figures) by the matrix
\[ \Scale[0.7]{ H_T = \left(
\begin{array}{cccccccc}
 1.59851 & 22.1297 & 315.941 & 50.4751 & 4410.11 & -618.164 & -1277.19 & 604.661 \\
 22.1297 & 308.294 & 4400.11 & 702.244 & 61426.0 & -8598.70 & -17795.7 & 8380.20 \\
 315.941 & 4400.11 & 62804.1 & 10023.2 & 876747. & -122731. & -253994. & 119635. \\
 50.4751 & 702.244 & 10023.2 & 1600.20 & 139922. & -19594.8 & -40535.0 & 19115.7 \\
 4410.11 & 61426.0 & 876747. & 139922. & 1.22394\times 10^7 & -1.71329\times 10^6 & -3.54581\times 10^6 & 1.66992\times 10^6 \\
 -618.164 & -8598.70 & -122731. & -19594.8 & -1.71329\times 10^6 & 239945. & 496327. & -234115. \\
 -1277.19 & -17795.7 & -253994. & -40535.0 & -3.54581\times 10^6 & 496327. & 1.02737\times 10^6 & -483535. \\
 604.661 & 8380.20 & 119635. & 19115.7 & 1.66992\times 10^6 & -234115. & -483535. & 229541. \\
\end{array}
\right). } \]
We apply Algorithm \ref{alg_LLL_for_SO_2g} with $\delta = 0.9$ to the matrix $H_T$; this yields the element
\[ \Scale[0.7]{ \gamma = \left(
\begin{array}{cccccccc}
 4 & 0 & -13 & -3 & 378 & 18 & -22 & 342 \\
 2 & -4 & 58 & -2 & -811 & 1 & -15 & -870 \\
 0 & 0 & 24 & 0 & -392 & 0 & 1 & -402 \\
 -1 & 14 & -185 & 1 & 2793 & -1 & 1 & 2992 \\
 0 & 0 & -2 & 0 & 32 & 0 & 0 & 33 \\
 0 & 1 & -13 & 0 & 200 & 0 & 0 & 214 \\
 0 & 0 & -1 & 0 & 14 & 0 & 0 & 15 \\
 0 & 0 & 0 & 0 & 1 & 0 & 0 & 1 \\
\end{array}
\right) \in \SO_8(\bZ),} \]
which has the property that ${}^t \gamma H_T \gamma = H_{\gamma^{-1} T \gamma}$ is $\delta$-reduced. We finally compute
\[ \Scale[0.7]{ \gamma^{-1} T \gamma = \left(
\begin{array}{cccccccc}
 0 & \frac{3}{2} & -\frac{1}{2} & 1 & 2 & 0 & \frac{1}{2} & -2 \\
 -1 & 1 & 1 & 1 & 1 & -1 & 2 & \frac{1}{2} \\
 0 & -1 & -1 & 0 & 0 & 0 & -1 & 0 \\
 0 & 0 & 1 & 0 & 0 & 0 & 1 & 2 \\
 0 & 0 & -1 & 0 & 0 & 0 & 1 & 1 \\
 0 & \frac{1}{2} & -1 & -1 & 1 & -1 & 1 & -\frac{1}{2} \\
 0 & 0 & \frac{1}{2} & 0 & 0 & -1 & 1 & \frac{3}{2} \\
 0 & 0 & 0 & 0 & 0 & 0 & -1 & 0 \\
\end{array}
\right).} \]
Ideally we would like to generalise this example by allowing elements of $V(\bZ)$ corresponding to arbitrary rational points of the Jacobian of the hyperelliptic curve $Y^2 = f(X)$. It is a relatively simple matter to write down elements of $V(\mathbb{Q})$ corresponding to such points (for example, by using a construction similar to the one given in \cite{Tho14} for odd hyperelliptic curves), but showing these elements are $\bG(\mathbb{Q})$-conjugate to points of $V(\bZ)$ requires a minimisation algorithm generalising those given in \cite{Cre10} for elliptic curves. Such an algorithm has been given for curves of genus 2 by Fisher--Liu \cite{Fis23b}, although in this case our algorithm is not needed due to the existence of an exceptional isomorphism $\mathrm{PSO}_6 \cong \PGL_4$ (see \cite[Remark 4.3]{Fis23a}), a phenomenon that does not persist when $2g > 6$.

\section{The case $\bG = G_2$}\label{sec_case_G_2}

We finally make Algorithm \ref{alg_LLL_for_G} explicit in the case of the exceptional group of type $G_2$, which can be realised as the group of automorphisms of the split octonions. We first give a description of this group scheme and its associated symmetric space in terms of octonions, following \cite{Gan03} (see also \cite{KMRT} and \cite{SpringerVeldkamp}). Let $\mathbb{A} = M_2(\bZ)$, and let $\ast$ denote the (anti-)involution of $\bA$ that sends a matrix $X$ to its adjugate $X^\ast$. The Cayley--Dickson realisation of the split octonions is $\bO = \bA \oplus \bA$, equipped with the multiplication
\[ (x, y) \cdot (z, w) = (xz - w y^\ast, x^\ast w + z y). \]
As is well-known, multiplication in $\bO$ is neither commutative nor associative, but is alternative (i.e. satisfies the identities $(a a) b = a ( ab)$ and $(a b) b = a ( b b)$ for all $a, b \in \bO$). We define a quadratic form $q$ on $\bO$ by the formula $q(x, y) = \det(x) + \det(y)$. This quadratic form is multiplicative, in the sense that for all $a, b \in \bO$, we have $q(ab) = q(a) q(b)$. 

We define  $\bG = G_2 \leq \GL(\bO)$ to be the group scheme over $\bZ$ of automorphisms of the triple $(\bO, q, \cdot)$. Thus for any ring $R$, $\bG(R) \leq \GL(\bO \otimes_\bZ R)$ is the group of automorphisms of the base extension $(\bO \otimes_\bZ R, q_R, \cdot_R)$.  Then $\bG$ is a subgroup of the split reductive group $\SO(q)$, and is a split semisimple group of type $G_2$ \cite[\S 6]{Gan03}. 

We fix a $\bZ$-basis of $\bO$ with respect to which the matrix of $q$ equals the matrix $S$ of (\ref{eqn_matrix_of_symmetric_form}). This will determine an embedding $\bG \to \SO_8$ into the group considered in \S \ref{sec_case_SO_2g}. Such a basis is given by the elements $e_1, \dots, e_8 \in \bO$, where
\[ e_1 = \left( \left(\begin{array}{cc} 0 & 0 \\ 1 & 0 \end{array}\right), 0 \right),\; e_2= \left( 0, \left(\begin{array}{cc} 0 & 0 \\ 1 & 0 \end{array}\right) \right), \; e_3 = \left(0,  \left(\begin{array}{cc} 0 & 0 \\ 0 & 1 \end{array}\right) \right),\; e_4 = \left( \left(\begin{array}{cc} 0 & 0 \\ 0& 1 \end{array}\right), 0 \right), \]
\[ e_5 = \left( \left(\begin{array}{cc} 1 & 0 \\ 0 & 0 \end{array}\right), 0 \right), \; e_6= \left(0, \left(\begin{array}{cc} 1 & 0 \\ 0 & 0 \end{array}\right) \right), \; e_7 = \left(0,  \left(\begin{array}{cc} 0 & -1\\ 0 & 0\end{array}\right) \right), \; e_8 = \left( \left(\begin{array}{cc} 0 & -1 \\ 0& 0 \end{array}\right), 0 \right). \]
A split maximal torus in $\bG$ is given by the intersection $\bT$ of $\bG$ with the diagonal maximal torus of $\SO_8$. Explicitly, this is
\[ \bT = \{ \diag(st, s, t, 1, 1, 1/t, 1/s, 1/(st)) \mid s, t \in \bG_m \}. \]
Given $x = \diag(st, s, t, 1, 1, 1/t, 1/s, 1/(st))$, a root basis is $\Delta = \{ \alpha_1(x) = t, \alpha_2(x) = s / t\}$. The associated system of positive roots is given by
\[ \Phi^+ = \{ \alpha_1, \alpha_2, \alpha_3(x) = s, \alpha_4(x) = st, \alpha_5(x) = s t^2, \alpha_6(x) = s^2 t \}. \]
The simple coroots are defined by
\[ \check{\alpha}_1(t) = \diag(t, 1/t, t^2, 1, 1, 1/t^2, t, 1/t), \,\,\check{\alpha}_2(t) = \diag(1, t, 1/t, 1, 1, t, 1/t, 1). \]
The group $\bN$ is the intersection of the corresponding subgroup of $\SO_8$  with $\bG_2$: equivalently, the group of automorphisms of $\bO$ that are unipotent upper-triangular in the basis $e_1, \dots, e_8$. 

We have thus far defined $A = \bT(\bR)^\circ$ and $N = \bN(\bR)$. The Cartan involution $\theta_0$ of Proposition \ref{prop_special_Cartan_involution} is the restriction of the Cartan involution $x \mapsto {}^t x^{-1}$ of $\SO_8(\bR)$ (which leaves $G = G_2(\bR)$ stable). Its fixed points form a maximal compact subgroup $K \leq G_2(\bR)$. We can now state the following proposition, giving a concrete interpretation of the symmetric space of $\bG$:
\begin{proposition}\label{prop_interpretation_of_symmetric_space_G_2}
There is a $G_2(\bR)$-equivariant bijection between the following two sets:
\begin{enumerate} 
\item The set $X_{G_2(\bR)} = K \backslash G$.
\item The set of inner products $H$ on $\bO_\bR$ that are compatible with $q$ (in the sense that $H^{-1} ({}^t S) = S^{-1} ({}^t H)$) and with $\cdot$ (in the sense that for any $x, y \in \bO_\bR$, we have $\| x \cdot y \|_H \leq \| x \|_H \| y \|_H$).
\end{enumerate}
\end{proposition}
\begin{proof}
This follows from \cite[Corollary 14.5]{Gan03}. The base-point of $X_{G_2(\bR)}$ corresponds to the inner product $H_0$ whose Gram matrix (in the basis $e_1, \dots, e_8$) is the $8 \times 8$ identity matrix (noting that \cite[Proposition 14.3]{Gan03} shows that this indeed defines an inner product compatible with $q$ and  $\cdot$).
\end{proof}
The Iwasawa decomposition gives us a bjiection
\[ A \times N \to X_{G_2(\bR)}, (a, n) \mapsto K a n. \]
It follows that any inner product $H$ on $\bO_\bR$ compatible with $q$ and $\cdot$ can be expressed as ${}^t (an) an$ for a unique pair $(a, n) \in A \times N$. Just as in the classical cases above, this can be computed using the Gram--Schmidt orthogonalisation process.  Let $e^\ast_{1}, \dots, e^\ast_8$ denote the result of carrying out this process on the standard basis $e_1, \dots, e_8$, using the inner product $H$, and let $\mu_{i, j} = (e_j, e_i^\ast)_H/(e_i^\ast, e_i^\ast)_H$. 
\begin{lemma}\label{lem_Gram_Schmidt_G_2}
With notation as above, let $n = n_H = ( \mu_{i, j} )_{1 \leq i, j \leq 8}$, $a = a_H = \diag( \| e_i^\ast \|_H)_{1 \leq i \leq 8}$. Then $a \in A$, $n \in N$, and $H = {}^t(an) an$. 
\end{lemma}
\begin{proof}
The proof is essentially the same as for Lemma \ref{lem_Gram_Schmidt}, since we have chosen the embedding $G_2 \leq \SO_8$ so that the groups $A, N \leq G_2(\bR)$ are contained inside the corresponding subgroups of $\SO_8(\bR)$.
\end{proof}
We now discuss size reduction. Let $\omega$ denote the set of elements $n \in N$ such that $n_{1, 2}, n_{2, 3}, n_{1, 3}, n_{1, 4}, n_{1, 6}$, and $n_{1, 7}$ all lie in $[-1/2, 1/2]$.
\begin{proposition}\label{prop_size_reduction_G_2}
The set $\omega$ is an allowable fundamental set. 
\end{proposition}
\begin{proof}
We first show that for any $n \in N$, there exists $\gamma_n \in \bN(\bZ)$ such that $n \gamma_n \in \omega$. This can be seen using the usual `column operation' process, \emph{given} the explicit form of the groups $u_{\alpha}(t)$ corresponding to a choice of positive root vectors $X_\alpha$. Here are the explicit formulae with respect to one such choice:
\[  \Scale[0.7]{u_{\alpha_1}(t) = \left(
\begin{array}{cccccccc}
 1 & t & 0 & 0 & 0 & 0 & 0 & 0 \\
 0 & 1 & 0 & 0 & 0 & 0 & 0 & 0 \\
 0 & 0 & 1 & t & -t & t^2 & 0 & 0 \\
 0 & 0 & 0 & 1 & 0 & t & 0 & 0 \\
 0 & 0 & 0 & 0 & 1 & -t & 0 & 0 \\
 0 & 0 & 0 & 0 & 0 & 1 & 0 & 0 \\
 0 & 0 & 0 & 0 & 0 & 0 & 1 & -t \\
 0 & 0 & 0 & 0 & 0 & 0 & 0 & 1 \\
\end{array}
\right), \enspace
u_{\alpha_2}(t) = \left(
\begin{array}{cccccccc}
 1 & 0 & 0 & 0 & 0 & 0 & 0 & 0 \\
 0 & 1 & t & 0 & 0 & 0 & 0 & 0 \\
 0 & 0 & 1 & 0 & 0 & 0 & 0 & 0 \\
 0 & 0 & 0 & 1 & 0 & 0 & 0 & 0 \\
 0 & 0 & 0 & 0 & 1 & 0 & 0 & 0 \\
 0 & 0 & 0 & 0 & 0 & 1 & -t & 0 \\
 0 & 0 & 0 & 0 & 0 & 0 & 1 & 0 \\
 0 & 0 & 0 & 0 & 0 & 0 & 0 & 1 \\
\end{array}
\right),\enspace
u_{\alpha_3}(t) = \left(
\begin{array}{cccccccc}
 1 & 0 & t & 0 & 0 & 0 & 0 & 0 \\
 0 & 1 & 0 & -t & t & 0 & t^2 & 0 \\
 0 & 0 & 1 & 0 & 0 & 0 & 0 & 0 \\
 0 & 0 & 0 & 1 & 0 & 0 & -t & 0 \\
 0 & 0 & 0 & 0 & 1 & 0 & t & 0 \\
 0 & 0 & 0 & 0 & 0 & 1 & 0 & -t \\
 0 & 0 & 0 & 0 & 0 & 0 & 1 & 0 \\
 0 & 0 & 0 & 0 & 0 & 0 & 0 & 1 \\
\end{array}
\right), } \]
\[  \Scale[0.7]{ u_{\alpha_4}(t) = \left(
\begin{array}{cccccccc}
 1 & 0 & 0 & t & -t & 0 & 0 & t^2 \\
 0 & 1 & 0 & 0 & 0 & -t & 0 & 0 \\
 0 & 0 & 1 & 0 & 0 & 0 & t & 0 \\
 0 & 0 & 0 & 1 & 0 & 0 & 0 & t \\
 0 & 0 & 0 & 0 & 1 & 0 & 0 & -t \\
 0 & 0 & 0 & 0 & 0 & 1 & 0 & 0 \\
 0 & 0 & 0 & 0 & 0 & 0 & 1 & 0 \\
 0 & 0 & 0 & 0 & 0 & 0 & 0 & 1 \\
\end{array}
\right),\enspace
u_{\alpha_5}(t) = \left(
\begin{array}{cccccccc}
 1 & 0 & 0 & 0 & 0 & t & 0 & 0 \\
 0 & 1 & 0 & 0 & 0 & 0 & 0 & 0 \\
 0 & 0 & 1 & 0 & 0 & 0 & 0 & -t \\
 0 & 0 & 0 & 1 & 0 & 0 & 0 & 0 \\
 0 & 0 & 0 & 0 & 1 & 0 & 0 & 0 \\
 0 & 0 & 0 & 0 & 0 & 1 & 0 & 0 \\
 0 & 0 & 0 & 0 & 0 & 0 & 1 & 0 \\
 0 & 0 & 0 & 0 & 0 & 0 & 0 & 1 \\
\end{array}
\right),\enspace
u_{\alpha_6}(t) = \left(
\begin{array}{cccccccc}
 1 & 0 & 0 & 0 & 0 & 0 & t & 0 \\
 0 & 1 & 0 & 0 & 0 & 0 & 0 & -t \\
 0 & 0 & 1 & 0 & 0 & 0 & 0 & 0 \\
 0 & 0 & 0 & 1 & 0 & 0 & 0 & 0 \\
 0 & 0 & 0 & 0 & 1 & 0 & 0 & 0 \\
 0 & 0 & 0 & 0 & 0 & 1 & 0 & 0 \\
 0 & 0 & 0 & 0 & 0 & 0 & 1 & 0 \\
 0 & 0 & 0 & 0 & 0 & 0 & 0 & 1 \\
\end{array}
\right).} \] 
It remains to check that $p(\omega) \subset [-1/2, 1/2]^{\Delta}$. This is true because, with the same choice of root vectors, we have $p(n) = (n_{1, 2}, n_{2, 3})$.
 \end{proof}
 The simple reflections determined by the same choice of vectors $X_\alpha$ as in the proof of Proposition \ref{prop_size_reduction_G_2} are given by
 \[ \Scale[0.7]{ s_1 = \left(
\begin{array}{cccccccc}
 0 & 1 & 0 & 0 & 0 & 0 & 0 & 0 \\
 -1 & 0 & 0 & 0 & 0 & 0 & 0 & 0 \\
 0 & 0 & 0 & 0 & 0 & 1 & 0 & 0 \\
 0 & 0 & 0 & 0 & 1 & 0 & 0 & 0 \\
 0 & 0 & 0 & 1 & 0 & 0 & 0 & 0 \\
 0 & 0 & 1 & 0 & 0 & 0 & 0 & 0 \\
 0 & 0 & 0 & 0 & 0 & 0 & 0 & -1 \\
 0 & 0 & 0 & 0 & 0 & 0 & 1 & 0 \\
\end{array}
\right), 
\quad 
s_2 = \left(
\begin{array}{cccccccc}
 1 & 0 & 0 & 0 & 0 & 0 & 0 & 0 \\
 0 & 0 & 1 & 0 & 0 & 0 & 0 & 0 \\
 0 & -1 & 0 & 0 & 0 & 0 & 0 & 0 \\
 0 & 0 & 0 & 1 & 0 & 0 & 0 & 0 \\
 0 & 0 & 0 & 0 & 1 & 0 & 0 & 0 \\
 0 & 0 & 0 & 0 & 0 & 0 & -1 & 0 \\
 0 & 0 & 0 & 0 & 0 & 1 & 0 & 0 \\
 0 & 0 & 0 & 0 & 0 & 0 & 0 & 1 \\
\end{array}
\right). } \]
We are now ready to give a more explicit version of Algorithm \ref{alg_LLL_for_G} for the group $G_2$.
\begingroup
\begin{algorithm}
\caption{Lattice reduction for $G_2$.} 
\begin{flushleft}
\textbf{Input:} A constant $\delta \in (1/4, 1)$, and an inner product $H$ on $\bO_\bR$ compatible with $q$ and $\cdot$.

\textbf{Output:} An element $\gamma \in G_2(\bZ)$ such that ${}^t \gamma H \gamma$ corresponds to a $\delta$-reduced point of $X_{G_2}$ under the bijection of Proposition \ref{prop_interpretation_of_symmetric_space_G_2}.
\end{flushleft}
  \begin{algorithmic}[1]\label{alg_LLL_for_G_2}
  \STATE Using Gram--Schmidt, compute $a \in A, n \in N$ such that $H = {}^t (an) an$. Set $a' \leftarrow a$, $n' \leftarrow n$, $\gamma \leftarrow 1 \in G_2(\bZ)$. 
\STATE Execute RED($1$). 
\IF{$\alpha_1(a')^2 > (\delta - (n'_{1, 2})^2)^{-1}$}
\STATE Execute REFL($1$).
\STATE \textbf{go to} 2.
\ENDIF
\STATE Execute RED($2$). 
\IF{$\alpha_2(a')^2 > (\delta - (n'_{2, 3})^2)^{-1}$}
\STATE Execute REFL($2$).
\STATE \textbf{go to} 2.
\ELSE
\FOR{$i = 3$ \TO $6$}
\STATE Execute RED($i$).
\ENDFOR
\ENDIF
\RETURN $\gamma$.
\end{algorithmic}
\end{algorithm}
\begin{algorithm}
\caption{Subroutine RED($i, j$) for Algorithm \ref{alg_LLL_for_G_2}.}
\begin{algorithmic}
\STATE Let $(k, l)$ be the $i^\text{th}$ tuple in the list $(1, 2), (2, 3), (1, 3), (1, 4), (1, 6), (1, 7),$
and let $m$ be the unique integer such that $n'_{k, l} + m \in [-1/2, 1/2)$. Set $n' \leftarrow n' u_{\alpha_i}(-m)$, $\gamma \leftarrow \gamma u_{\alpha_i}(m)$. 
\end{algorithmic}
\end{algorithm}
\begin{algorithm}
\caption{Subroutine REFL($i$) for Algorithm \ref{alg_LLL_for_G_2}.}
\begin{algorithmic}
\STATE Let $(k, l)$ be the $i^\text{th}$ tuple in the list $(1, 2), (2, 3)$. Set $t \leftarrow n_{k, l}$, $m' \leftarrow u_{\alpha_i}(-t) n'$, $n' \leftarrow u_{\alpha_i}( (-t \alpha_i(a')^2 )/(1 + t^2 \alpha_i(a')^2) )s_i^{-1} m' s_i$, $a' \leftarrow \check{\alpha}_i(\sqrt{1 + t^2 \alpha_i(a')^2}) s_i(a')$, $\gamma \leftarrow \gamma s_i$.
\end{algorithmic}
\end{algorithm}
\endgroup

We conclude with a numerical example. We randomly generated an inner product $H$ on $\bO_\bR$ compatible with $q$ and $\cdot$, given to 3 decimal places in the basis $e_1, \dots, e_8$ by the matrix 
\[ \Scale[0.7]{ H =\left(
\begin{array}{cccccccc}
 677.821 & 742.548 & 658.571 & -370.423 & 370.423 & -52.366 & 856.644 & -685.136 \\
 742.548 & 918.869 & 682.461 & -550.939 & 550.939 & -158.762 & 1100.78 & -807.478 \\
 658.571 & 682.461 & 660.727 & -299.161 & 299.161 & -5.648 & 761.899 & -647.973 \\
 -370.423 & -550.939 & -299.161 & 410.997 & -409.997 & 177.778 & -703.984 & 448.569 \\
 370.423 & 550.939 & 299.161 & -409.997 & 410.997 & -177.778 & 703.984 & -448.569 \\
 -52.366 & -158.762 & -5.648 & 177.778 & -177.778 & 113.394 & -236.83 & 102.243 \\
 856.644 & 1100.78 & 761.899 & -703.984 & 703.984 & -236.83 & 1351.24 & -947.165 \\
 -685.136 & -807.478 & -647.973 & 448.569 & -448.569 & 102.243 & -947.165 & 725.913 \\
\end{array}
\right). } \]
Running the algorithm with $\delta = 0.9$ produces an element $\gamma \in G_2(\bZ) \leq \SO_8(\bZ)$, given by the matrix
\[ \Scale[0.7]{ \gamma = \left(
\begin{array}{cccccccc}
 -7 & 11 & -15 & -6 & 6 & 18 & 10 & -28 \\
 14 & -9 & 18 & 3 & -3 & -33 & -11 & 36 \\
 -6 & -4 & -1 & 5 & -5 & 11 & -3 & -4 \\
 10 & 0 & 8 & -3 & 4 & -21 & -1 & 18 \\
 -10 & 0 & -8 & 4 & -3 & 21 & 1 & -18 \\
 2 & -16 & 14 & 12 & -12 & -9 & -14 & 23 \\
 10 & -10 & 16 & 4 & -4 & -24 & -9 & 31 \\
 4 & -14 & 14 & 10 & -10 & -14 & -14 & 25 \\
\end{array}
\right). } \]
The image of $H$ under $\gamma$ is
\[ \Scale[0.7]{ {}^t \gamma H \gamma = \left(
\begin{array}{cccccccc}
 0.802 & 0.282 & 0.318 & -0.215 & 0.215 & -0.287 & 0.026 & 0.162 \\
 0.282 & 1.013 & -0.165 & -0.536 & 0.536 & -0.018 & 0.266 & 0.052 \\
 0.318 & -0.165 & 1.088 & 0.363 & -0.363 & -0.03 & -0.425 & 0.297 \\
 -0.215 & -0.536 & 0.363 & 1.398 & -0.398 & 0.426 & -0.66 & -0.227 \\
 0.215 & 0.536 & -0.363 & -0.398 & 1.398 & -0.426 & 0.66 & 0.227 \\
 -0.287 & -0.018 & -0.03 & 0.426 & -0.426 & 1.511 & 0.035 & -0.793 \\
 0.026 & 0.266 & -0.425 & -0.66 & 0.66 & 0.035 & 1.734 & -0.437 \\
 0.162 & 0.052 & 0.297 & -0.227 & 0.227 & -0.793 & -0.437 & 1.909 \\
\end{array}
\right).} \]

\section*{Appendix: size reduction for a general $\bG$}

We now return to the notation of \S \ref{s-algorithm}, so $\bG$ is an arbitrary split semisimple group. In this appendix,
we define coordinates on $N$, describe an allowable fundamental set $\omega \subset N$, and describe size reduction in this setting, i.e. given $n \in N$, we show how to find $\gamma \in \bN(\bZ)$ such that $n\gamma \in \omega$. 

Recall that we have fixed a basis vector $X_\al$ for each $\al \in \Delta$. We now choose a basis vector $X_\al$ for each $\al \in \Phi^+$. These basis vectors determine root-group homomorphisms $u_\al: \bG_a \to \bG$ for $\al \in \Phi^+$. We say that an ordering $\alpha_1, \dots, \alpha_r$ of the elements of $\Phi^+$ is \emph{good} if it satisfies the following property:
\begin{center}
Suppose $i, j, k \in \{1, \dots, r\}$ such that $\al_k = c_i\al_i + c_j\al_j$ for some integers $c_i, c_j \geq 1$. Then $k > \max\{i, j\}$. 
\end{center}
Note that good orderings exist. For example, given $\be = \sum_{\al \in \Delta} c_\al\al \in \Phi^+$, we define the height of $\be$ to be $\Ht(\be) = \sum_{\al \in \Delta} c_\al$. Then we can take $\alpha_1, \dots, \alpha_r$ to be any ordering of $\Phi^+$ with the property that if $i \leq j$ then $\Ht(\alpha_i) \leq \Ht(\alpha_j)$.
For another example, consider $\bG = \SL_{n}$ with its diagonal maximal torus $\bT$ 
and roots $\al_{i,j}$ as defined in \S \ref{subsec_setup_in_case_SL_g}. 
We claim that 
\[ \alpha_{1, 2}, \alpha_{2, 3}, \alpha_{1, 3}, \al_{3,4}, \al_{2, 4}, \al_{1, 4}, \dots, \alpha_{i-1, i}, \al_{i - 2, i}, \dots, \alpha_{1, i}, \dots, \alpha_{n-1, n}, \al_{n - 2, n} \dots, \alpha_{1, n} \]
is a good ordering. Indeed, suppose that $i < j$, that  $k < \ell$, and that $c, d \in \bZ$ are positive integers. If $c\al_{i, j} + d\al_{k, \ell} \in \Phi$, we must have $c = d = 1$, $j = k$, and $\al_{i, \ell} = \al_{i, j} + \al_{k, \ell}$. Thus $j < \ell$, so $\al_{i, j}$ appears before $\al_{i, \ell}$ in this ordering. And $i < k$, so $\al_{k, \ell}$ appears before $\al_{i, \ell}$ in this ordering.

For the rest of the appendix, fix a good ordering $\al_1, \dots, \al_r$ of the positive roots. By \cite[Corollary 2 to Lemma 18]{Steinberg}, every element $n \in N$ can be written uniquely in the form $n = u_{\al_1}(t_1)\dots u_{\al_r}(t_r)$ for some $t_1, \dots, t_r \in \bR$. We may think of $(t_1, \dots, t_r)$ as coordinates for $n$. Given such an $n \in N$, we write $t_i(n)$ for $t_i$ in these coordinates. Similarly, every $\gamma \in \bN(\bZ)$ may be written uniquely in the form $u_{\al_1}(m_1)\dots u_{\al_r}(m_r)$ for some $m_1, \dots, m_r \in \bZ$. 
Our choice of a good ordering has the following implications:

\begin{lemma}\label{lem_nilpotent_group}
Given $k \in \{1, \dots, r\}$, let $N_k = \{u_{\al_k}(t_k)u_{\al_{k + 1}}(t_{k + 1})\dots u_{\al_r}(t_r) \mid t_k, \dots, t_r \in \bR\}$.
\begin{enumerate} 
\item Each $N_k$ is a normal subgroup of $N$. 
\item We have $[N, N_k] \leq N_{k+1}$. 
\item If $1 \leq i \leq k$, then $t_i(n)$ depends only on $n \text{ mod }N_{k+1}$, and the map $u_{\alpha_1} \times \dots \times u_{\alpha_k} : \bR^k \to N / N_{k+1}$ is bijective.
\end{enumerate}
\end{lemma}
\begin{proof}
Let $\Phi_k = \{\al_i \mid i \geq k\}$. By the definition of a good ordering, we see that $\Phi_k$ is an ideal in $\Phi^+$ for all $k \in \{1, \dots, r\}$, in the sense of \cite[Chapter 3]{Steinberg}. By \cite[Lemma 17]{Steinberg}, we see that $N_k$ is the subgroup generated by the root groups $x_\al(\bR)$ for $\al \in \Phi_k$, and by \cite[Lemma 16]{Steinberg}, the subgroup $N_k$ is normal in $N$. The rest of the lemma then follows from commutator relations for root groups. More specifically, the second statement follows from \cite[Corollary to Lemma 15]{Steinberg}. 
The third statement follows from uniqueness of expression in \cite[Lemma 17]{Steinberg}. 
\end{proof}

\begin{remark}
Suppose the simple roots appear first in our ordering $\al_1, \dots, \al_r$ (this occurs, for example, in an ordering that respects height as mentioned above), and let $k = \# \Delta$. In this case 
 the group $N_+$ mentioned in \S \ref{s-reduced} is equal to $N_{k + 1}$. The map $p$ in that section sends an element $n$ to its first $k$ coordinates with respect to this ordering, and the map $p_\al$ gives the coordinate corresponding to $\al \in \Delta$. 
 \end{remark}

Let $\omega = \{ n \in N \mid \forall i \in \{1, \dots, r\}, t_i(n) \in [-1/2, 1/2) \}$. The next proposition will show that $\omega$ is an allowable fundamental set in $N$, and the proof will demonstrate the process of size reduction in this setting.

\begin{proposition}\label{prop_size_reduction}
Let $n \in N$. Then there is a unique element $\gamma_n \in \bN(\bZ)$ such that $n \gamma_n \in \omega$.
\end{proposition}

\begin{proof}
Let $n = u_{\al_1}(t_1)\dots u_{\al_r}(t_r)$. We will inductively choose integers $m_1, \dots, m_r$ such that $nu_{\al_1}(m_1)\dots u_{\al_r}(m_r) \in \omega$. First let $m_1$ be the unique integer such that $t_1 + m_1 \in [-1/2, 1/2)$. By Lemma \ref{lem_nilpotent_group}(2) $u_{\al_1}(m_1)$ commutes with $u_{\al_2}(t_2)\dots u_{\al_r}(t_r)$ modulo $N_2$, so $nu_{\al_1}(m_1)N_2 = u_{\al_1}(t_1)u_{\al_1}(m_1)N_2 = u_{\al_1}(t_1 + m_1)N_2$ in $N/N_2$. By Lemma \ref{lem_nilpotent_group}(3), we have $t_1(nu_{\al_1}(m_1)) = t_1 + m_1 \in [-1/2, 1/2)$.

For the induction step, assume that we have chosen $m_1, \dots, m_k \in \bZ$ such that $t_i(nu_{\al_1}(m_1)u_{\al_2}(m_2)\dots u_{\al_k}(m_k)) \in [-1/2, 1/2)$ for $i \in \{1, \dots, k\}$. Let $n' = nu_{\al_1}(m_1)u_{\al_2}(m_2)\dots u_{\al_k}(m_k)$, and write $n'$ as $u_{\al_1}(t_1')\dots u_{\al_r}(t_r')$. Let $m_{k + 1}$ be the unique integer such that $t_{k + 1}' + m_{k + 1} \in [-1/2, 1/2)$. Using the same logic as above, by Lemma \ref{lem_nilpotent_group}, we have $n'u_{\al_{k + 1}}(m_{k + 1})N_{k + 2} = u_{\al_1}(t_1')\dots u_{\al_{k + 1}}(t_{k + 1}')u_{\al_{k + 1}}(m_{k + 1})N_{k + 2}$ in $N/N_{k + 2}$, so $t_{k + 1}(n'u_{\al_{k + 1}}(m_{k + 1})) = t_{k + 1}' + m_{k + 1} \in [-1/2, 1/2)$, and $t_j(n'u_{\al_{k + 1}}(m_{k + 1})) = t_j(n')$ for all $j \in \{1, \dots, k\}$. 
In this way, we see that we may choose $m_1, \dots, m_r$ such that $nu_{\al_1}(m_1)\dots u_{\al_r}(m_r) \in \omega$. Thus we have proven the existence of $\gamma_n = u_{\al_1}(m_1)\dots u_{\al_r}(m_r)$.

Next we show uniqueness. Assume $\gamma, \gamma' \in \bN(\bZ)$ satisfy $n\gamma, n\gamma' \in \omega$. Write $\gamma = u_{\al_1}(m_1)\dots u_{\al_r}(m_r)$ and $\gamma' = u_{\al_1}(m_1')\dots u_{\al_r}(m_r')$ (where $m_i, m_i' \in \bZ$ for all $i$). We will show inductively that $m_i = m_i'$ for all $i$. Using the same logic as in the first part of the proof, we see that $n\gamma N_2 = u_{\al_1}(t_1 + m_1)N_2$ and $n\gamma'N_2 = u_{\al_1}(t_1 + m_1')N_2$ in $N/N_2$. By the definition of $\omega$, we must have $t_1 + m_1 \in [-1/2, 1/2)$ and $t_1 + m_1' \in [-1/2, 1/2)$, which implies $m_1 = m_1'$. 

Now assume that $m_j = m_j'$ for all $j \in \{1, \dots, k\}$. By Lemma \ref{lem_nilpotent_group}, the integer $m_{k+1}$ (resp. $m'_{k+1}$) must have the property that
\[ t_{k+1}( u_{\alpha_{1}}(t_{1}) \dots u_{k + 1}(t_{k + 1}) u_{\alpha_1}(m_1) \dots u_{\alpha_{k+1}}(m_{k+1})  ) \]
lies in $[-1/2, 1/2)$ (resp. the same with $m_{k+1}$ replaced by $m'_{k+1}$). This condition uniquely characterises $m_{k+1}$ and does not involve the $m_i$ with $i \geq k+2$, so forces $m_{k+1} = m'_{k+1}$. 
\end{proof}

Notice that to use the method of this proof to compute $n\gamma_n$ and $\gamma_n$ in coordinates, one must have explicit commutator formulas for root groups. These can be computed inductively (cf. \cite[Chapter 3]{Steinberg}). If $\bG$ is adjoint, then commutators are described explicitly in \cite[Section 5.2]{Carter}. 

\bibliographystyle{alpha}
\bibliography{LLL-split}

\end{document}